\documentclass{amsart}
\usepackage{amsmath}
\usepackage{amssymb}
\usepackage{array}
\usepackage[matrix,arrow,curve]{xy}
\usepackage{hyperref}
\usepackage{latexsym}

    \numberwithin{equation}{section}
     \newtheorem{theorem}{Theorem}[section]
\newtheorem{lemma}[theorem]{Lemma}

   \theoremstyle{definition}
   \newtheorem{definition}[theorem]{Definition}

   \theoremstyle{remark}
    \newtheorem{remark}[theorem]{Remark}

         \theoremstyle{notation}
    \newtheorem{notation}[theorem]{Notation}

         \theoremstyle{corollary}
    \newtheorem{corollary}[theorem]{Corollary}

\newcommand{\Equat}{\operatorname{Equat}}
\newcommand{\degg}{\operatorname{deg}}
\newcommand{\ord}{\operatorname{ord}}
\newcommand{\Pic}{\operatorname{Pic}}
\newcommand{\Aut}{\operatorname{Aut}}
\newcommand{\Cr}{\operatorname{Cr}}
\newcommand{\Ker}{\operatorname{Ker}}
\newcommand{\Stab}{\operatorname{Stab}}

\newcommand{\Imm}{\operatorname{Im}}
\newcommand{\Disc}{\operatorname{Disc}}
\newcommand{\Orb}{\operatorname{Orb}}
\newcommand{\elm}{\operatorname{elm}}

\title{The conjugacy classes of finite nonsolvable subgroups in the plane Cremona group.}
\author{Vladimir Igorevich Tsygankov}

\address{Johannes Gutenberg-Universit\"{a}t, Institut f\"{u}r Mathematik, Staudingerweg 9, 55099 Mainz, Germany}
\email{vova.tsygankov@googlemail.com}
\thanks{The research leading to these results has received partial funding from the European Union Seventh Framework Programme (FP7/2007-2013) under grant agreement 248826 and from the grant NSh-4731.2010.1.}

\begin{document}
    \maketitle

\begin{abstract}
  The aim of this paper  is to give a finer geometric description of the algebraic varieties parametrizing
conjugacy classes of nonsolvable subgroups in the plane Cremona group.
\end{abstract}



    \section{Introduction.}

    The classification of finite subgroups in the plane Cremona group over the field $\mathbb{C}$ denoted by $\Cr_2(\mathbb{C})$ is a classical problem.  The history of this problem begins with the work of E. Bertini \cite{3}, where are classified the conjugacy classes of subgroups of order $2$ in $\Cr_2(\mathbb{C})$. There were obtained  three families of conjugacy classes now called as involution de Jonqui\`{e}res, Geiser and Bertini. However, the classification was  incomplete, and the proof was not rigorous. Only recently in \cite{1} was obtained complete and short proof.

    In 1895 Kantor \cite{7} and Wiman \cite{8} gave a description of finite subgroups in $\Cr_2(\mathbb{C})$. The list was fairly comprehensive, but was not full in the following aspects. Firstly, for a given finite subgroup on this list could not be defined, whether it is contained in the Cremona group or not. Secondly, the question of  conjugacy between the subgroups was not considered.

    Modern approach to the problem was initiated by the work of Manin \cite{Manin-67} and continued in  works of Iskovskikh \cite{9}, \cite{10}, \cite{11}. In the paper \cite{Manin-67} is established a clear link between the classification of conjugacy classes of finite subgroups of the Cremona group and the classification of $G$-minimal rational surfaces $ (S, G) $ and $G$-equivariant birational maps between them. The consideration is divided into two cases: when $S$ is a del Pezzo surface, and when $S$ is a conic bundle.

\begin{definition}
 Let $G$ be a finite group.  A $G$-surface is a triple  $(S,G,\rho)$, where $S$ is a nonsingular projective surface, and $\rho$ is a monomorphism of the group $G$ to the automorphisms group of the surface $S$. For brevity, $G$-surface will be denoted by $(S, G)$.
     \end{definition}
Let $G$ be a finite subgroup in $\Cr_2(\mathbb{C})$ with an embedding $\theta: G \hookrightarrow \Cr_2(\mathbb{C})$. It turns out that the action of $G$ on $\mathbb{P}^2$ can be \textit{regularized}, i.e there exists a smooth rational surface $S$ and a birational map $\mu: S \dashrightarrow \mathbb{P}^2$ such that $\mu^{-1} \circ \theta(G) \circ \mu$ is a subgroup of automorphisms of $S$.

   Certainly, any regularization is not unique. For example, if we blow up any $G$-orbit of points on $S$. Two distinct $G$-surfaces $(S,G)$ and $(S',G)$ define two conjugate embeddings $\theta: G \rightarrow \Cr_2(\mathbb{C})$ and $\theta': G\rightarrow \Cr_2(\mathbb{C})$ respectively, iff there exist a $G$-equivariant birational map $\zeta: S \dashrightarrow S'$.

\begin{definition}
A $G$-surface $(S,G)$ is called  $G$-minimal, if any $G$-equivariant birational morphism $S\rightarrow Y$ onto a smooth $G$-surface $Y$ is a $G$-isomorphism.
    \end{definition}
\begin{theorem}[(\cite{Manin-67})]
  \label{Man1}
 There are two types of the rational $G$-minimal surfaces $(S,G)$:
                      \begin{itemize}
             \item $S$ is a del Pezzo surface, and $\Pic(S)^G \simeq \mathbb{Z}$;
          \item         $S$ has a $G$-equivariant structure of conic bundle $\phi: S \rightarrow \mathbb{P}^1$, and $\Pic(S)^G \simeq \mathbb{Z}^2$.
                                  \end{itemize}
                      \end{theorem}
               \begin{notation}
             The classes of $G$-minimal rational surfaces from  the first and the second cases of the Theorem \ref{Man1} will be denoted respectively as $\mathbb{D}$ and $\mathbb{CB}$.
               \end{notation}

    More recently, I.V.~Dolgachev and V.A.~Iskovskikh \cite{Dolgachev-Iskovskikh} improved the list of Kantor and Wiman.  The answer was obtained in terms of action of the groups $G$ on the del Pezzo surfaces and on the conic bundles. It was considered question about conjugacy of the finite subgroups in $\Cr_2(\mathbb{C})$, using the theory of elementary links of V.A.~Iskovskikh (see \cite{Iskovskikh-96}). For general case this paper is currently the most precise classification of conjugacy classes of finite subgroups in $\Cr_2(\mathbb{C})$. I note that J.~Blanc in \cite{Blanc-2011} obtained a more precise classification in case of finite cyclic subgroups.

However in  \cite{Dolgachev-Iskovskikh} explicit  equations of $G$-minimal surfaces $(S,G)$ in weighted projective spaces and  explicit descriptions of  actions of  $G$ on  surfaces $S$ were obtained only in case of Del Pezzo surfaces. Also description of groups $G$, acting on $G$-minimal conic bundles $(S,G,\phi)$, was given only in terms of groups extensions. In other words, for a given abstract finite group it is still impossible, using \cite{Dolgachev-Iskovskikh}, to say whether the group is isomorphic to a subgroup of $\Cr_2(\mathbb{C})$. Also classification of conjugacy classes of finite subgroups in $\Cr_2(\mathbb{C})$ has some gaps. If $G \subset \Cr_2(\mathbb{C})$ is regularized as a subgroup of automorphisms of a conic bundle $\phi: S \rightarrow \mathbb{P}^1$ with $K_S^2=1$, or $2$, and $(S,G,\phi)\in \mathbb{CB}$. I will show it consistently for $K_S^2=1$ and 2.

 Let $K_S^2=1$. Consider \cite[Section 8.1, Pages 534-535]{Dolgachev-Iskovskikh}. There is stated non-existence of triples $(S,G,\phi)\in \mathbb{CB}$ with $K_S^2=1$ and ample divisor $-K_S$. However, it's wrong. An example of such triples is presented in \cite[Section 6.2.3, Theorem 6.8]{Tsygankov-10}.  In  this case the authors of \cite{Dolgachev-Iskovskikh} applied an old incorrect version of \cite[Theorem 5.7]{Dolgachev-Iskovskikh}. This version existed until J.~Blanc reported about a mistake to I.~Dolgachev. I note that in published version of \cite{Dolgachev-Iskovskikh} Theorem 5.7 is presented in correct form. Unfortunately, for large volume of work, the authors forgot to update some conclusions from the theorem.

Let $K_S^2=2$. In \cite[Section 8.1, Pages 535]{Dolgachev-Iskovskikh} is stated: if $(S,G,\phi)\in \mathbb{CB}$ with $K_S^2=2$ and non-ample divisor $-K_S$ then   the surface $S$ is \textit{exceptional conic bundle}  (see Definition \ref{defExceptcon}). In other words the surface $S$ contains two smooth non-intersecting rational $(-3)$-curves. This is also wrong. In \cite[Section 5.1.1, Theorem 5.4]{Tsygankov-10} is presented an example of triple $(S,G,\phi)\in \mathbb{CB}$ with $K_S^2=2$ and nef, non-ample divisor $-K_S$.

    In the paper  \cite{Tsygankov-10} I continue classification of $G$-minimal conic bundles, which was begun in \cite{Dolgachev-Iskovskikh}.
For given arbitrary value of $K_S^2$, it was constructed a method of classification by means of explicit  equations of $G$-minimal conic bundles $(S,G)$ in weighted projective spaces and  explicit descriptions of the actions of  $G$ on the Picard group $\Pic(S) $ and on the surface $S$. The classification is carried on completely for $K_S^2> 0$. If $K_S^2 \le 0$ then the $G$-minimal conic bundle $(S,G)$ is birationally rigid. So there is no question about conjugacy  (see \cite[Section 8]{Dolgachev-Iskovskikh}).

    The aim of this paper  is to give a finer geometric description of the algebraic varieties parameterizing
conjugacy classes of finite nonsolvable subgroups in $\Cr_2(\mathbb{C})$, applying methods of papers \cite{Dolgachev-Iskovskikh} and \cite{Tsygankov-10}. It is obtained explicit  equations of $G$-minimal surfaces $(S,G)$ in weighted projective spaces and  explicit descriptions of  actions of  $G$ on  surfaces $S$. Also all possibilities for the groups $G$ are fully described.

In \cite[Section 9]{Dolgachev-Iskovskikh} were stated the following problems for $\Cr_2(\mathbb{C})$:
\begin{itemize}
   \item Find the finer classification of the conjugacy classes of de Jonqui\`{e}res groups.
    \item Give a finer geometric description of the algebraic varieties parameterizing conjugacy classes.
\end{itemize}
This article gives a solution of these problems for the nonsolvable finite subgroups in $\Cr_2(\mathbb{C})$.

It's important to note that investigation method described in the paper can be employed to solve these problems for all finite subgroups in $\Cr_2(\mathbb{C})$, i.e. not necessary nonsolvable. However due to large amount of routine work investigation was conducted only for nonsolvable subgroups.

 The paper has the following structure. In Section \ref{DelPezzo} I study  surfaces from the class $\mathbb{D}$, i.e. the $G$-minimal del Pezzo surfaces $(S,G)$, where  $\Pic(S)^G\simeq \mathbb{Z}$ and $G$ is a finite nonsolvable group.
I will apply here  results of \cite{Dolgachev-Iskovskikh}. There are no my results in this section.

In Section \ref{ConBundle} I study  surfaces from the class $\mathbb{CB}$, i.e. the $G$-minimal surfaces $(S,G,\phi)$, where a morphism $\phi: S\rightarrow \mathbb{P}^1$ defines a $G$-equivariant conic bundle structure, $\Pic(S)^G\simeq \mathbb{Z}^2$, and $G$ is a finite nonsolvable group.
The main  my results are described in this section.

In Section \ref{Conjugate} I study conjugacy classes of embeddings $G\rightarrow \Cr_2(\mathbb{C})$ defined by $G$-minimal surfaces $(S,G)$, for all finite nonsolvable subgroups $G \subset \Cr_2(\mathbb{C})$. Here I reprove results in \cite[Section 7]{Dolgachev-Iskovskikh} for the sake of completeness.

     In Section \ref{List} I present a list of the finite nonsolvable subgroups in $\Cr_2(\mathbb{C})$, obtained on the basis of results of sections \ref{DelPezzo} and \ref{ConBundle}.

  This work is dedicated  to my supervisor Vasily Alekseevich Iskovskikh, who initiated  my study of the Cremona group. I am very grateful to Yuri Gennadievich Prokhorov and Ilya Alexandrovich Tyomkin for useful tips and remarks.

   The base field is assumed everywhere to be $\mathbb{C}$.
 Throughout this paper we will use the following  notations.
  \begin{itemize}
     \item $\varepsilon_n$ denotes a primitive n-th root of unity.
   \item $S_n$  denotes the permutation group of degree $n$.
   \item $A_n$ denotes the alternating group of degree $n$.
      \item Consider a subgroup $A_5\subset PGL(2,\mathbb{C})$, which is isomorphic to the icosahedral automorphisms group, and the standard projection $\psi: SL(2,\mathbb{C})\rightarrow PGL(2,\mathbb{C})$. Then $\bar{A}_5$ denotes the group $\psi^{-1}(A_5)$, which is isomorphic to the binary icosahedral group.
     \item $A.B$, where $A$ and $B$ are some abstract groups, is one of the possible extensions with help of the exact sequence: $0\rightarrow A \rightarrow G \rightarrow B \rightarrow 0$.
    \item Let $H$ be an abstract group. Then $H\wr S_n$ will denote the semidirect product $H^n \rtimes S_n$, where $S_n$ is the symmetric group, acting on $H^n$ by permuting the factors.
        \item $A \triangle_D B$ is a diagonal product of abstract groups $A$ and $B$ over their common homomorphic image  $D$ (i.e. the subgroup of $A \times B$ of pairs $(a,b)$, such that $\alpha(a)=\beta(b)$ for some epimorphisms $\alpha: A \rightarrow D,\ \beta: B \rightarrow D$).
       \item $\mathbb{P}(a_1,\ldots,a_n)$, where $a_i \in \mathbb{Z}$, $i=1,\ldots,n$, is the weighted projective space, with the set of weights $(a_1,\ldots,a_n)$.
   \end{itemize}

\section{Case of del Pezzo surfaces.}
\label{DelPezzo}
  In this section we study the surfaces $(S,G)\in \mathbb{D}$, i.e. $S$ is a $G$-minimal del Pezzo surface, and $\Pic(S)^G\simeq \mathbb{Z}$. The groups $G$ are supposed to be finite nonsolvable.
  We will apply here  results of \cite{Dolgachev-Iskovskikh}. There are no author's results in this section.

Recall that a surface $S$ is called a del Pezzo surface, if $S$ is smooth, and $-K_S$ is ample. It's well known that $1 \le K_S^2 \le 9$. We will carry our investigation, considering  different values of $K_S^2$.

 In the next theorem we study the case $K_S^2=9$. In this case $S \simeq \mathbb{P}^2$.
   \begin{theorem}
\label{K_S9}
      Let $(S,G) \in \mathbb{D}$, $K_S^2=9$, and $G$ be a finite nonsolvable group. Then $S \simeq \mathbb{P}^2$ with the coordinates $(x_0:x_1:x_2)$, and $G$ is any finite nonsolvable subgroup of $\Aut(\mathbb{P}^3)\simeq PGL(3,\mathbb{C})$. The subgroup $G \subset PGL(3,\mathbb{C})$ can be conjugated to one of the following subgroups.
     \begin{enumerate}
          \item $H$ is a  group, consisting of maps
          $$(x_0:x_1:x_2)\mapsto (ax_0+bx_1:cx_0+dx_1:x_2).$$
    The image of matrices
   $$\begin{pmatrix}
         a & b\\
            c & d
   \end{pmatrix}
\in GL(2, \mathbb{C})     $$
    in $PGL(2,\mathbb{C})$ under the natural projection $GL(2,\mathbb{C}) \rightarrow PGL(2,\mathbb{C})$ is isomorphic to $A_5$. The group $H$ is isomorphic to  $\mathbb{Z}_n \times \bar{A}_5$, $n\ge 1$.

    \item The icosahedral group $A_5$ isomorphic to $L_2(5)$. It leaves invariant a nonsingular conic $C \subset \mathbb{P}^2$.
     \item The Klein group isomorphic to $L_2(7)$. This group is realized as automorphism group of the Klein's quartic      $x_0^3x_1+x_1^3x_2+x_2^3x_0=0$.
        \item The Valentiner group isomorphic to $A_6$. It can be realized as automorphism group of the nonsingular plane sextic
   $$
    10x_0^3x_1^3+ 9x_2x_0^5 + x_1^6-45x_0^2 x_1^2 x_2^2 -135 x_0x_1 x_2^4 +27 x_2^6=0.
     $$
 \end{enumerate}
    \end{theorem}
 \begin{proof}
   The statement follows directly from \cite[Corollary 4.6, Theorems 4.7, 4.8]{Dolgachev-Iskovskikh}. We need only to check the isomorphism $H \simeq \mathbb{Z}_n \times \bar{A}_5$, $n\ge 1$ in the first case of  theorem. It follows from \cite[Lemma 4.5, case (i)]{Dolgachev-Iskovskikh}.
\end{proof}

 In the next theorem we consider the case $K_S^2=8$.
 \begin{theorem}
  \label{K_S8}
      Let $(S,G) \in \mathbb{D}$, $K_S^2=8$, and $G$ be a finite nonsolvable group. Then $S \simeq \mathbb{F}_0 \simeq \mathbb{P}^1 \times \mathbb{P}^1$ with the coordinates $(x_0:x_1,t_0:t_1)$. We will employ definition of the group $St(A_5)$ (see Notation \ref{stA5}), and define involution
    $$\tau: (x_0:x_1,t_0:t_1)\mapsto (t_0:t_1, x_0:x_1).$$
We have the following possibilities for $G$.
            \begin{enumerate}
  \item The subgroup $G \subset \Aut(\mathbb{P}^1 \times \mathbb{P}^1)$ is conjugate to the subgroup $St(A_5)\wr \langle \tau \rangle$.
   \item The subgroup $G \subset \Aut(\mathbb{P}^1 \times \mathbb{P}^1)$ is conjugate to the subgroup $H\times \langle \tau \rangle$, where $H$ is the image of the diagonal embedding of $St(A_5)$ in $PGL(2,\mathbb{C}) \times PGL(2,\mathbb{C})$.
      \end{enumerate}
  \end{theorem}
              \begin{proof}
  One knows that if $S$ is a del Pezzo surface with $K_S^2=8$ then $S \simeq \mathbb{F}_0$ or $\mathbb{F}_1$. However in the second case the exceptional section of ruled surface $\mathbb{F}_1$ is $G$-invariant. Therefore the pair $(\mathbb{F}_1,G)$ is not $G$-minimal. Hence $S\simeq \mathbb{F}_0 \simeq \mathbb{P}^1 \times \mathbb{P}^1$.

     It's well known that $\Aut(\mathbb{P}^1 \times \mathbb{P}^1)\simeq  PGL(2,\mathbb{C}) \wr \langle \tau \rangle$. Whence $G$ is generated by a nonsolvable subgroup $H \subset PGL(2,\mathbb{C}) \times PGL(2,\mathbb{C})$ and by an element $\eta=\mu \circ \tau$, where $\mu \in PGL(2,\mathbb{C}) \times PGL(2,\mathbb{C})$.  We write $\mu=(B,B')$, where $B$, $B'\in PGL(2,\mathbb{C})$. For any $\varsigma= (A,A')\in PGL(2,\mathbb{C}) \times PGL(2,\mathbb{C})$ we have
  \begin{equation}
   \label{teh1}
         \eta \circ \varsigma \circ \eta^{-1}=(BA'B^{-1} , B'AB'^{-1}).
   \end{equation}
Let's study the structure of  group $H$, applying Goursat's Lemma (see \cite[Lemma 4.1]{Dolgachev-Iskovskikh}). Consider projections $\pi_i: PGL(2,\mathbb{C}) \times PGL(2,\mathbb{C}) \rightarrow PGL(2,\mathbb{C})$, $i=1,2$ on the first and the second factor respectively. We get $H \simeq \pi_1(H) \triangle_D \pi_2(H)$, where $D \simeq \Imm(\pi_1|H)/\Ker(\pi_2|H)$. Obviously, either $\Imm(\pi_1|H) \simeq A_5$ or $\Imm(\pi_2|H)\simeq A_5$ (see Klein's classification of finite nonsolvable subgroups in $PGL(2,\mathbb{C})$ in \cite[Section 5.5]{Dolgachev-Iskovskikh}).
The group $A_5$ is simple. Therefore $D \simeq 1$ or $A_5$.

   Suppose that $D\simeq 1$. From \eqref{teh1} we get $\Imm(\pi_1|H) \simeq \Imm(\pi_2|H)\simeq A_5$. Hence $H$ can be conjugated to $St(A_5) \times St(A_5)$. We will prove that $B$, $B'\in St(A_5)$. Suppose that it doesn't holds.
Then from  \eqref{teh1} we get $St(A_5) \rtimes B$, $St(A_5) \rtimes B' \subset PGL(2,\mathbb{C})$. But $St(A_5)$ is a maximal finite subgroup of $PGL(2, \mathbb{C})$. Contradiction. We get the first case of the theorem.

   Suppose that $D\simeq A_5$. Then $H$ is conjugated to the image of  diagonal embedding of $St(A_5)$ in $PGL(2,\mathbb{C}) \times PGL(2,\mathbb{C})$. Arguing as above, we get $B$, $B' \in St(A_5)$. By \eqref{teh1} these elements define the same inner automorphism of $A_5$. Hence $B=B'$. We get the second case of the theorem.
             \end{proof}

 In the next theorem we consider  cases: $K_S^2=7$, $K_S^2=6$, $K_S^2=4$, and $K_S^2=1$.
   \begin{theorem}
  \label{thdp3}
    There are no  surfaces $(S,G)\in \mathbb{D}$, such that  $K_S^2$ is equal to either $7$, or $6$, or $4$ , or $1$,  and $G$ is a finite nonsolvable group.
      \end{theorem}
   \begin{proof}
  Let's consider the case $K_S^2=7$. The surface $S$ is presented as a blowing up of two different points in $\mathbb{P}^2$. However the strict transform of line, containing this two points, is a $G$-invariant rational $(-1)$-curve. Hence the surface $S$ is not $G$-minimal.

   The case $K_S^2=6$ follows directly from \cite[Theorem 6.3, Corollary 4.6, Theorem 4.7]{Dolgachev-Iskovskikh}.

    The case $K_S^2=4$ follows directly from \cite[Theorem 6.9]{Dolgachev-Iskovskikh}.

    The case $K_S^2=1$ follows directly from \cite[Table 8]{Dolgachev-Iskovskikh}.
\end{proof}
        In the next theorem we consider the case $K_S^2=5$.
  \begin{theorem}
 \label{thdp4}
     Let $(S,G) \in \mathbb{D}$, $K_S^2=5$, and $G$ be a finite nonsolvable group. Introduce on $\mathbb{P}^2$ the coordinates $(T_0:T_1:T_2)$.
Then the surface $S$ is isomorphic to the blowing up of $\mathbb{P}^2$ at points: $(0:0:1)$, $(0:1:0)$, $(1:0:0)$ and $(1:1:1)$.
The group $\Aut(S)$ is isomorphic to $S_5$, and is generated by the  maps:
     \begin{equation}
    \label{K_S5eq}
         \begin{aligned}
      &(T_0:T_1:T_2)\mapsto (T_1:T_2:T_0),\\
        &(T_0:T_1:T_2)\mapsto (T_2:-T_0+T_2:-T_1+T_2),\\
        &(T_0:T_1:T_2)\mapsto (T_0(T_2-T_1):T_2(T_0-T_1):T_0T_2).
             \end{aligned}
              \end{equation}
The subgroup $G \subset \Aut(S)$ is isomorphic to $A_5$ or $S_5$.
   \end{theorem}
  \begin{proof}
         This follows directly from \cite[Theorem 6.4]{Dolgachev-Iskovskikh} and arguments of \cite[Theorem 8.4.15]{DolgachevTopics}.
\end{proof}
        In the next theorem we consider the case $K_S^2=3$.
     \begin{theorem}
   \label{thdp2}
 Let $(S,G) \in \mathbb{D}$, $K_S^2=3$, and $G$ be a finite nonsolvable group. Then the surface $S$ can be represented by the following equations in $\mathbb{P}^3$ with the coordinates $(T_0:T_1:T_2:T_3)$:
        $$T_0^2 T_1+T_1^2T_2+T_2^2T_3+T_3^2T_0=0. $$
 The group $G$ is isomorphic to $S_5$ and is generated by the following maps:
              \begin{equation}
        \label{K_S3eq}
       \begin{aligned}
    &(T_0:T_1:T_2:T_3)\mapsto (T_0:\varepsilon_5^4 T_1:\varepsilon_5 T_2:\varepsilon_5^2 T_3),\\
    &(T_0:T_1:T_2:T_3)\mapsto (T_1:T_2:T_3:T_0).
   \end{aligned}
   \end{equation}
    \end{theorem}
  \begin{proof}
  This follows directly from \cite[Theorem 6.14]{Dolgachev-Iskovskikh}.
   \end{proof}
  In the next theorem we consider the case $K_S^2=2$.
     \begin{theorem}
\label{thdp1}
 Let $(S,G) \in \mathbb{D}$, $K_S^2=2$, and $G$ be a finite nonsolvable group. Then the surface $S$ can be represented by  the following equation in $\mathbb{P}(2,1,1,1)$ with the coordinates $(T_0:T_1:T_2:T_3)$:
   \begin{equation}
        \label{K_S2}
   T_3^2+T_0^3T_1+T_1^3T_2+T_2^3T_0=0.
    \end{equation}
  The group $\Aut(S)$ is isomorphic to $\mathbb{Z}_2 \times L_2(7)$.
    The subgroup $G \subset \Aut(S)$ is isomorphic to either  $L_2(7)$, or $\mathbb{Z}_2 \times L_2(7)$.
\end{theorem}
   \begin{proof}
   This follows directly from \cite[Theorem 6.17]{Dolgachev-Iskovskikh}.
  \end{proof}

    \section{Case of conic bundles}
   \label{ConBundle}
    In this section we study the surfaces $(S,G,\phi)$ in the class $\mathbb{CB}$, i.e. $G$-minimal surfaces $S$ with $\Pic(S)^G\simeq \mathbb{Z}^2$, having a $G$-equivariant conic bundle structure $\phi: S\rightarrow \mathbb{P}^1$. The groups $G$ are supposed to be finite nonsolvable.

     Recall (see \cite[Item 3.7]{Dolgachev-Iskovskikh}) that a rational $G$-surface $(S,G)$ has a conic bundle structure, if there exist a $G$-equivariant morphism $\phi: S \rightarrow \mathbb{P}^1$,   whose each fiber $F_t=\phi^{-1}(t),\ t\in \mathbb{P}^1$ is either a nondegenerate plane conic (isomorphic to $\mathbb{P}^1$) or a reducible reduced conic, i.e. a pair of intersecting lines.

     A conic bundle $(S,G,\phi)$ is said to be \emph{relatively $G$-minimal}, if the fibres of $\phi$ do not contain  $G$-orbits, consisting of nonintersecting rational $(-1)$-curves (i.e. components of   reducible fibres --- equivalently to  $\Pic(S)^G=\phi^*\Pic(\mathbb{P}^1)\oplus \mathbb{Z}\simeq \mathbb{Z}^2$). Recall that a $G$-surface $(S,G)$ is said to be \emph{$G$-minimal}, if any $G$-equivariant birational morphism $S\rightarrow Y$ onto a smooth $G$-surface $Y$ is a $G$-isomorphism.  It is clear that a $G$-minimal surface, having a conic bundle structure, is relatively minimal.   The inverse statement is not always valid.

Denote by $r$ the number of the reducible fibers of a conic bundle $(S,G,\phi)$.
By Noether formula we have $r=8-K_S^2$, so $K_S^2\le 8$. If $K_S^2=8$, then $S$ is  isomorphic to Hirzebruch's surface $\mathbb{F}_n$, $n\ge 0$.

          \begin{theorem}
        \label{th01}
      Let $(S,G,\phi) \in \mathbb{CB}$  with $K_S^2=8$, and $G$ be a finite nonsolvable group. Then the surface $S$ is isomorphic to Hirzebruch's surface $\mathbb{F}_n$, $n\ge 0$. The morphism $\phi: S\rightarrow \mathbb{P}^1$ coincides with the standard projection $\mathbb{F}_n \rightarrow \mathbb{P}^1$.
            \begin{enumerate}
           \item Let $n=0$. Then $\mathbb{F}_0\simeq \mathbb{P}^1\times \mathbb{P}^1$. The group $\Aut(\mathbb{F}_0)$ is isomorphic to $PGL(2,\mathbb{C}) \wr S_2$.
     The subgroup $G\subset PGL(2,\mathbb{C})\times PGL(2,\mathbb{C}) \subset  \Aut(\mathbb{F}_0)$ is isomorphic to one of the following: $A_5 \times B$,   $B \times A_5$,  $A_5 \triangle_{A_5} A_5$, where $B$ is any finite subgroup of $PGL(2, \mathbb{C})$.
          \item  Let $n>0$. Then $n>1$. Consider $\mathbb{F}_n \rightarrow \mathbb{P}(n,1,1)$  the blowdown of  exceptional section of $\mathbb{F}_n$. We have
$$\Aut(\mathbb{F}_n)\simeq \mathbb{C}^{n+1} \rtimes (GL(2,\mathbb{C}) \slash \mu_n),$$
where $GL(2,\mathbb{C}) \slash \mu_n$ acts on $\mathbb{C}^{n+1}$ by means of its natural linear representation in the space of binary forms with degree n.
 The subgroup $G\subset \Aut(\mathbb{F}_n)$ is isomorphic to one of the following groups:
\begin{equation}
            \label{dff}
     G \simeq
     \left \{
      \begin{aligned}
      &\mathbb{Z}_m \times A_5,  m \ge 1, & \textrm{if $n$ is even};\\
      &\mathbb{Z}_m \times \bar{A}_5,  m \ge 1, & \textrm{if $n$ is odd}.
      \end{aligned}
 \right.
    \end{equation}
        \end{enumerate}
        \end{theorem}
               \begin{proof}
   If $S\simeq \mathbb{F}_0 \simeq \mathbb{P}^1\times \mathbb{P}^1$, then $G$ is a nonsolvable subgroup in $\Aut(\mathbb{F}_0)\simeq   PGL(2,\mathbb{C}) \wr S_2$. Note that $G \subset  PGL(2,\mathbb{C}) \times PGL(2,\mathbb{C})$, so as $\Pic(S)^G \simeq \mathbb{Z}^2$.
We apply  Goursat's lemma (see \cite[Lemma 4.1]{Dolgachev-Iskovskikh}) and Klein's classification of the finite subgroups in $PGL(2,\mathbb{C}$) (see \cite[Section 5.5]{Dolgachev-Iskovskikh}).  We get that $G  \simeq B \triangle_D C$, where one of groups $B$ and $C$ is isomorphic to $A_5$. Since the group $A_5$ is simple, the group $D$ is isomorphic to either $1$ or $A_5$. Therefore the group $G$ is isomorphic to one of the following groups:
 $A_5 \times B$,   $B \times A_5$,  $A_5 \triangle_{A_5} A_5$, where $B$ is any finite subgroup of $PGL(2, \mathbb{C})$.
Remark that a group $A_5 \triangle_{A_5} A_5$ is conjugate to image of a diagonal embedding of group $A_5$ to $PGL(2,\mathbb{C}) \times PGL(2,\mathbb{C})$.

   Consider the case $S \simeq \mathbb{F}_n,\ n > 0$. Let $\mathbb{F}_n \rightarrow \mathbb{P}(n,1,1)$ be the blowdown of the exceptional section of $\mathbb{F}_n$.

We note that if $n=1$ then $\mathbb{P}(1,1,1)$ is a smooth surface. Hence the triple $(S,G,\phi)$ is not minimal.
Therefore $n\ne 1$.

Introduce the coordinates $(x:t_0:t_1)$ on $\mathbb{P}(n,1,1)$. It's well known (see \cite[Theorem 4.10]{Dolgachev-Iskovskikh}) that $\Aut(\mathbb{F}_n)\simeq \mathbb{C}^{n+1} \rtimes (GL(2,\mathbb{C}) \slash \mu_n)$. The group $\mathbb{C}^{n+1}$ is generated by maps $(x:t_0:t_1) \mapsto (x+f_n(t_0,t_1):t_0:t_1)$, where $f_n$ is a binary form with degree $n$. The group $GL(2,\mathbb{C}) \slash \mu_n$ is generated by invertible maps $(x:t_0:t_1) \mapsto (x:at_0+bt_1:ct_0+dt_1)$. Moreover, we have
    \begin{equation*}
     GL(2,\mathbb{C}) \slash \mu_n \simeq
     \left \{
      \begin{aligned}
     & \mathbb{C}^* \rtimes SL(2,\mathbb{C}), &\text{if $n$ is odd};\\
     & \mathbb{C}^* \rtimes PGL(2,\mathbb{C}), &\text{if $n$ is even}.
      \end{aligned}
 \right.
    \end{equation*}
Consider the sequence of homomorphisms
$G \xrightarrow{h_1} GL(2,\mathbb{C}) \slash \mu_n \xrightarrow{h_2} PGL(2,\mathbb{C})$, where $h_2$ is
natural projection. Obviously that
 the homomorphism $h_1$ is injective, and $\Imm(h_2) \simeq A_5$ (see \cite[Section 5.5]{Dolgachev-Iskovskikh}).

   Thus the group $G$ is isomorphic to a central extension $\mathbb{Z}_m.A_5$, if $n$ is even,
   or $\mathbb{Z}_m.\bar{A}_5$, if $n$ is odd. Applying \cite[Lemma 4.4]{Dolgachev-Iskovskikh}, we get \eqref{dff}.
            \end{proof}

 There are no relatively $G$-minimal conic bundles $(S,G,\phi)$, if $K_S^2=7$. So there are no $G$-minimal conic bundles too.
 If $K_S^2=3$, $5$ or $6$, then there exists a $G$-equivariant  morphism $(S,G)\rightarrow (S',G)$, where $(S',G) \in \mathbb{D}$ (see \cite[Proposition 2.1, Theorem 4.1]{10} and, for example, \cite[Section 2]{Tsygankov-10}). Thus study of this cases reduces to study of $G$-minimal del Pezzo surfaces.
For other values $K_S^2=4,2,1,\ldots$ relatively $G$-minimal conic bundles are always $G$-minimal.

The morphism $\phi:S \rightarrow \mathbb{P}^1$ induces a homomorphism $\phi_*: G \rightarrow \Aut(\mathbb{P}^1)$. We have the following exact sequence
  \begin{equation}
    \label{1}
  1\rightarrow G_K\rightarrow G\rightarrow G_B\rightarrow 1,
\end{equation}
   where $G_K \simeq \Ker(\phi_*)$, and $G_B \simeq \Imm(\phi_*)$.

  Also consider the natural representation of group $G$ in the automorphisms group of lattice $\Pic(S)$. By $G_0$ we denote the kernel of this representation. The group $G_0$ fixes the divisor classes of the sections with negative self-intersection. Such sections obviously exist. Hence $G_0$ fixes it pointwisely. Considering one of these sections as a point on a general fibre, we conclude that $G_0$ is a cyclic group.

   \begin{theorem}[{\cite[Proposition 5.5]{Dolgachev-Iskovskikh}}]
 \label{th1.1.2}
 Let $(S,G,\phi) \in \mathbb{CB}$  with $K_S^2\le 4$, $K_S^2\ne3$. Suppose that $G_0\neq \{1\}$.
 Then  the surface  $S$ has an exceptional conic bundle structure (see below).
 \end{theorem}

   \begin{definition}
   \label{defExceptcon}
   Define the \emph{exceptional conic bundles}. This is the minimal resolution of singularities of  surface, given by the equation in  weighted projective space $\mathbb{P}(1,1,g+1,g+1)$, where $g\geq 1$:
 \begin{equation*}
 Y_g:F_{2g+2}(t_0,t_1)+t_2t_3=0,
 \end{equation*}
   \noindent where $F_{2g+2}$ is a binary form without multiple factors with degree $2g+2$.
\end{definition}

   After  the resolution of indeterminacy points,  the projection onto $\mathbb{P}^1$ with coordinates $(t_0,t_1)$ will induce a conic bundle structure $\phi:\widetilde{Y_g}\rightarrow \mathbb{P}^1$. This conic bundle has reducible fibres over the points from $\mathbb{P}^1$, where $F_{2g+2}(t_0,t_1)=0$. The surface  $\widetilde{Y_g}$ contains $2$   nonintersecting rational $(-g-1)$-curves defined  by the equations $t_2=0$ and $t_3=0$. Automorphisms of the surface $\widetilde{Y_g}$ are induced by automorphisms of $\mathbb{P}(1,1,g+1,g+1)$.

The case of exceptional conic bundles will be considered in Section \ref{Exceptional conic}.

 There is a theorem about the structure of minimal finite groups, acting on the non-exceptional conic bundles.

\begin{theorem}[{\cite[Theorem 5.7]{Dolgachev-Iskovskikh}}]
\label{th1.4}
Let $(S,G,\phi)\in \mathbb{CB}$  with $K_S^2\le 4$, $K_S^2\ne3$, and  $\Sigma$ be the set of reducible fibres of $\phi$. Suppose that $G_0\simeq 1$. Then one of the following cases occurs.
 \begin{enumerate}
 \item Case $G_K\simeq \mathbb{Z}_2$. The central involution  $\iota$,  generating the group $G_K$,
  fixes pointwise a smooth bisection  $C$ of the fibration $\phi$ and switches the components of fibres in a subset $\Sigma'\subset \Sigma$. The morphism  $\phi$ defines the linear system $g_2^1$ on the curve $C$ having  branch points in the singular points of the fibres in the set $\Sigma'$.
   Genus of the curve $C$ is equal to $g=(m-2)/2$, where $m=|\Sigma'|$. The group $G_B$ is isomorphic to a subgroup of the automorphism group of  curve $C$ modulo the involution defined by $g_2^1$.
 \item Case $G_K\simeq \mathbb{Z}_2^2$. Each nontrivial element $\iota_i,\ i=0,1,2$ of the group $G_K$ fixes pointwise
   an irreducible smooth bisection $C_i$. The set $\Sigma$ is partitioned into three subsets $\Sigma_0,\ \Sigma_1,\ \Sigma_2$, such that $\Sigma_i=(\Sigma_j \cup \Sigma_k)\setminus (\Sigma_j \cap \Sigma_k),\ i\neq j\neq k \ne i$.
The morphisms $\phi|_{C_i},\ i=0,1,2$ are branched over the singular points of  fibres in subsets $\Sigma_i$.
The group $G_B$ leaves invariant the set of points $\phi(\Sigma) \in \mathbb{P}^1$ and its partition into three subsets $\phi(\Sigma_i),\ i=0,1,2$.
 \end{enumerate}
 \end{theorem}

Consider cases of Theorem \ref{th1.4} separately. We note that the subgroup $G_B \subset \Aut(\mathbb{P}^1)\simeq PGL(2,\mathbb{C})$ (see \eqref{1}) is nonsolvable, since $G_K$ is solvable by Theorem \ref{th1.4}. Hence $G_B \simeq A_5$ (see Klein's classification of the finite subgroups in $PGL(2,\mathbb{C})$ in \cite[Section 5.5]{Dolgachev-Iskovskikh}). We will use the fact in sections \ref{conbund1}, \ref{conbund2}, \ref{conbund3} without mentioning.
\begin{enumerate}
\item Case $G_K \simeq \mathbb{Z}_2$, and $\Sigma'=\Sigma$. This case will be investigated in Section \ref{conbund1}.
\item Case $G_K \simeq \mathbb{Z}_2$, and $\Sigma' \ne \Sigma$. This case will be investigated in Section \ref{conbund2}.
\item Case $G_K \simeq \mathbb{Z}_2^2$. This case will be investigated in Section \ref{conbund3}.
\end{enumerate}

  We will often use the following facts about finite nonsolvable subgroups $\bar{P}\subset SL(2,\mathbb{C})$ (see \cite[Section 5.5]{Dolgachev-Iskovskikh}). Obviously that $\bar{P} \simeq \bar{A}_5$. Any group of this type is conjugated to a group with the following generators:
\begin{equation}
         \label{generat}
g_1 = \begin{pmatrix}\varepsilon_{10}&0\\
0&\varepsilon_{10}^{-1}\end{pmatrix}, \quad g_2 = \begin{pmatrix}0&i\\
i&0\end{pmatrix}, \quad g_3 = \frac{1}{\sqrt{5}}\begin{pmatrix}\varepsilon_5-\varepsilon_5^4&\varepsilon_5^2-\varepsilon_5^3\\
\varepsilon_5^2-\varepsilon_5^3&-\varepsilon_5+\varepsilon_5^4\end{pmatrix}.
\end{equation}
          \begin{notation}
\label{stA5}
   We will denote a group generated by \eqref{generat} as $St(\bar{A}_5)$. It's image in $PGL(2,\mathbb{C})$ we will denote as $St(A_5)$.
     \end{notation}
Consider the natural representation of $St(\bar{A}_5)$ in space of polynomials $\mathbb{C}[t_0,t_1]$. Space of relative invariants of the representation is generated by the following Gr\"undformens:
   \begin{equation}
   \label{grundform}
         \begin{aligned}
  &\Phi_1 = T_0^{30}+T_1^{30}+522(T_0^{25}T_1^5-T_0^{5}T_1^{25})-10005(T_0^{20}T_1^{10}+T_0^{10}T_1^{20}),\\
   &\Phi_2 = -(T_0^{20}+T_1^{20})+228(T_0^{15}T_1^5-T_0^{5}T_1^{15})-494T_0^{10}T_1^{10},\\
   &\Phi_3 = T_0T_1(T_0^{10}+11T_0^5T_1^5-T_1^{10}).
       \end{aligned}
   \end{equation}
Since $\bar{A}_5/(\pm 1) \cong A_5$ is a simple group and all Gr\"undformens have even degree, we easily see that $g(\Phi_i)=\Phi_i$, $i=1,2,3$, for any $g\in St(\bar{A}_5)$. In other words, the characters are trivial.
\begin{notation}
\label{mathcalI}
We will denote space of invariants of group $\bar{A}_5$ generated by this Gr\"undformens as $\mathcal{I}^{St(\bar{A}_5)}$.
\end{notation}

        \subsection{Case of exceptional conic bundles}
           \label{Exceptional conic}
   In this section we will prove the following theorem.
\begin{theorem}
   \label{thExcept}
Let $(S,G,\phi)\in \mathbb{CB}$ be an exceptional conic bundle, and $G$ be a finite nonsolvable group. Then the surface $S$ can be represented as the minimal resolution of singularities of surface given by  the equation in the weighted projective space $\mathbb{P}(1,1,g+1,g+1)$, $g\geq 1$ with  coordinates $(t_0:t_1:t_2:t_3)$:
$$
 Y_g:F_{2g+2}(t_0,t_1)+t_2t_3=0,
 $$
   \noindent where $F_{2g+2}\in \mathcal{I}^{St(\bar{A}_5)}$ is a binary form without multiple factors with degree $2g+2$. The morphism $\phi: S\rightarrow \mathbb{P}^1$ is induced by the map $\phi': Y_g \dashrightarrow \mathbb{P}^1$ given by
   $$\phi':\ (t_0:t_1:t_2:t_3)\mapsto (t_0:t_1). $$
The group $G$ is isomorphic to
   \begin{equation*}
    G\simeq \left \{
   \begin{aligned}
      & D_n \times \bar{A}_5,\ n\ge 1, \ \text{if $g$ is even};\\
      & D_n \times A_5,\ n\ge 1, \ \text{if $g$ is odd}.
   \end{aligned}
               \right.
   \end{equation*}
All possibilities for $G$ occur.
  The group $G$ is generated by the maps:
  \begin{equation}
 \label{3.6}
           \begin{aligned}
    & (t_0:t_1:t_2:t_3) \mapsto (\varepsilon_{10}t_0:\varepsilon_{10}^{-1}t_1:t_2:t_3),\\
  & (t_0:t_1:t_2:t_3) \mapsto (i t_1:i t_0:t_2:t_3),\\
  & (t_0:t_1:t_2:t_3) \mapsto ((\varepsilon_{5}-\varepsilon_{5}^4)t_0+(\varepsilon_{5}^2-\varepsilon_{5}^3)t_1:
 (\varepsilon_{5}^2-\varepsilon_{5}^3)t_0+(-\varepsilon_{5}+\varepsilon_{5}^4)t_1:t_2:t_3),\\
  & (t_0:t_1:t_2:t_3) \mapsto (t_0:t_1:\varepsilon_m t_2: \varepsilon_m^{-1} t_3),\\
  & (t_0:t_1:t_2:t_3) \mapsto (t_0:t_1:t_3:t_2),
   \end{aligned}
\end{equation}
 where $m=n$, if $g$ is odd, and $m=2n$, otherwise.
\end{theorem}
    \begin{proof}
   By \cite[Proposition 5.3]{Dolgachev-Iskovskikh} we have $\Aut{Y_g}\simeq N.P$, where $N \simeq \mathbb{C}^* \rtimes \mathbb{Z}_2 $ is  a group generated by the maps:
           \begin{equation}
        \label{3.7}
   \begin{aligned}
& (t_0:t_1:t_2:t_3)\mapsto (t_0:t_1:t_3:t_2),\\
& (t_0:t_1:t_2:t_3)\mapsto (t_0:t_1:ct_2: c^{-1}t_3),\ c\in \mathbb{C},\ c\ne 0.
    \end{aligned}
   \end{equation}
 And $P$ is the subgroup of $PGL(2,\mathbb{C})$, leaving the form $F_{2g+2}(t_0,t_1)$ semi-invariant. Obviously that $P\simeq A_5$. We conjugate the subgroup $P \subset PGL(2,\mathbb{C})$ to the subgroup $St(A_5)$.  Then $F_{2g+2} \in \mathcal{I}^{St(\bar{A}_5)}$. Whence we get that the group $\Aut(Y_g)$ is generated by maps \eqref{3.6} and \eqref{3.7}. Hence
  $$
\Aut(Y_g)\simeq \left \{
    \begin{aligned}
     (N\slash \mu_2) \times \bar{A}_5, \text{if $g$ is even};\\
     N \times A_5, \text{if $g$ is odd},
\end{aligned}
  \right.
   $$
   where the group $\mu_2$ acts by $(t_0:t_1:t_2:t_3)\mapsto (t_0:t_1:-t_2:-t_3)$.
   It follows from the description of exceptional conic bundles (see \cite[Section 5.2]{Dolgachev-Iskovskikh}) that the triple $(S,G,\phi)$ is minimal, iff the group $G$ permutes points: $(0:0:1:0)$ and $(0:0:0:1)$. Therefore $G\cap N \simeq D_n$. Further arguments are obvious.
   \end{proof}

   \subsection{Case, when $G_0\simeq 1$, $G_K \simeq \mathbb{Z}_2$, and $\Sigma'=\Sigma$.}
          \label{conbund1}
       Here we will apply arguments of \cite[Section 3.1]{Tsygankov-10}.

       The group $G_K$ is generated by involution $\iota$. By \cite[Theorem 3.2]{Tsygankov-10} we get $S/ \iota  \simeq \mathbb{F}_e$. The morphism $\pi: S \rightarrow S/ \iota \simeq \mathbb{F}_e$ is $G$-equivariant, and a faithful action of the group $G_B$ is defined on $\mathbb{F}_e$(see exact sequence \eqref{1}).
Recall (see Theorem \ref{th1.4}) that the morphism $\pi: S \rightarrow \mathbb{F}_e$ is branched over a nonsingular hyperelliptic curve $C$. Let $\bar{C}=\pi(C)$.
We consider cases $e=0$ and $e>0$ in Theorems \ref{th1} and \ref{th2} respectively.

  We make some preparations before statement of Theorem \ref{th1}. Introduce the coordinates $(x_0:x_1,t_0:t_1)$ on $\mathbb{F}_0\simeq \mathbb{P}^1 \times \mathbb{P}^1$. The morphism $\phi: S \rightarrow \mathbb{P}^1$ induces projection $(x_0:x_1,t_0:t_1) \mapsto (t_0:t_1)$. The curve $\bar{C}$ is represented by the equation:
     \begin{equation}
\label{eqC}
  \Equat(\bar{C})= p_0(t_0,t_1)x_0^2+2p_1(t_0,t_1)x_0x_1+p_2(t_0,t_1)x_1^2=0,
   \end{equation}
  where $p_i, \ i=0,1,2$ are binary forms with degree $2d$. Note that the degree is even, so as the divisor class $\bar{C}\in 2 \Pic(\mathbb{F}_0)$ (where $2 \Pic(\mathbb{F}_0)  \subset \Pic(\mathbb{F}_0) \simeq \mathbb{Z}^2$ is the even sublattice). The form $\Disc(\bar{C})=p_0p_2-p_1^2$ has no multiple factors, since $\bar{C}$ is nonsingular.  We will apply the Segre embedding $\nu: \mathbb{P}^1\times \mathbb{P}^1 \rightarrow \mathbb{P}^3$ to represent the surface $S$ by equations. Introduce the coordinates $(x:y:z:w)$ on $\mathbb{P}^3$. This embedding is given by $\nu: (x_0:x_1, t_0:t_1)\mapsto (x_0t_0:x_0t_1:x_1t_0:x_1t_1)$.
 We choose some polynomials $F_i(x,y,z,w)$, $i=0,\ldots,2d-2$, such that $x_0^i x_1^{2d-2-i}\Equat(\bar{C})=\nu^*(F_i)$. The surface $S$ is represented by the equations in $\mathbb{P}(d^{d},1^4)$ with the coordinates $u_i,\ x,\ y,\ z,\ w,\ i=0,\ldots,d-1$:
    \begin{equation}
  \label{sur1}
       \begin{aligned}
   & u_i u_j=F_{i+j},\ 0 \le i \le j \le d-1,\\
     & x^{j-i}u_i=u_j z^{j-i},\ y^{j-i}u_i=u_j w^{j-i},\ 0 \le i< j \le d-1,\\
      & xw=yz.
 \end{aligned}
   \end{equation}

    \begin{theorem}
  \label{th1}
  Let $(S,G,\phi)\in \mathbb{CB}$ , and $G$ be a finite nonsolvable group. Suppose that $G_0\simeq 1$, $G_K \simeq \langle \iota \rangle \simeq \mathbb{Z}_2$, $\Sigma'=\Sigma$, and $S/\iota \simeq \mathbb{F}_0$. Then the surface $S$ is represented by equations \eqref{sur1}.
The morphism $\phi: S \rightarrow \mathbb{P}^1$ is given by
$$\phi:\ (u_0:\ldots:u_{d-1}:x:y:z:w) \mapsto \left \{
   \begin{aligned}
   &(x:y), \text{if}\ (x:y)\ne (0:0);\\
   &(z:w), \text{if}\ (z:w)\ne (0:0).
    \end{aligned}
    \right.              $$

 There is defined a faithful action of $G_B$ (see exact sequence \eqref{1}) on $\mathbb{F}_0$, and $G_B \subset PGL(2,\mathbb{C})\times PGL(2,\mathbb{C})$.  One of the following cases occurs.
  \begin{enumerate}
     \item The subgroup $G_B \subset PGL(2,\mathbb{C}) \times PGL(2, \mathbb{C})$ is conjugated to the subgroup $1\times St(A_5)$, and $G \simeq \mathbb{Z}_2 \times A_5$.
     \item The subgroup $G_B \subset PGL(2,\mathbb{C}) \times PGL(2, \mathbb{C})$ is conjugated to the diagonal embedding $St(A_5)\hookrightarrow PGL(2,\mathbb{C}) \times PGL(2, \mathbb{C})$, and
     \begin{equation*}
      G \simeq \left \{
     \begin{aligned}
             &\bar{A}_5, \text{if $d$ is even in \eqref{eqC}},\\
              &  \mathbb{Z}_2 \times A_5, \text{otherwise}.
                                    \end{aligned}
                 \right.
     \end{equation*}
     \end{enumerate}
   All cases exist. And all possibilities for $G$ occur.

  In all cases  $G$ acts on $S$ by the following way. Embedding $G_B \simeq St(A_5)$ to $PGL(2,\mathbb{C}) \times PGL(2, \mathbb{C})$ defines a unique embedding $St(\bar{A}_5) \hookrightarrow SL(2,\mathbb{C}) \times SL(2, \mathbb{C})$. This defines an action of $St(\bar{A}_5)$ on the surface $S$ given by  equation \eqref{sur1}. The action of $St(\bar{A}_5)$ on coordinates $u_i,\ i=0,\ldots,d-1$ coincides with the action on monomials $x_0^i x_1^{d-1-i},\ i=0, \ldots, d-1$, respectively.
An action of $G$ is generated by the action of $St(\bar{A}_5)$ and by the map
 $$(u_0:\ldots:u_{d-1}:x:y:z:w) \rightarrow (-u_0:\ldots:-u_{d-1}:x:y:z:w).$$
 \end{theorem}
      \begin{proof}
 Recall that $G_B \simeq A_5$. The subgroup $G_B \subset \Aut(\mathbb{F}_0)$ acts trivially on $\Pic(\mathbb{F}_0)$. Hence $G_B \subset PGL(2,\mathbb{C}) \times PGL(2,\mathbb{C})$.  We apply Goursat's Lemma (see \cite[Lemma 4.1]{Dolgachev-Iskovskikh}) to study the subgroups $A_5 \subset PGL(2,\mathbb{C}) \times PGL(2,\mathbb{C})$. Consider projections $\pi_i: PGL(2,\mathbb{C}) \times PGL(2,\mathbb{C}) \rightarrow PGL(2,\mathbb{C})$, $i=1,2$ on the first and the second factor respectively. We get $G_B \simeq \pi_1(G_B) \triangle_D \pi_2(G_B)$, where $D \simeq \Imm(\pi_1|G_B)/\Ker(\pi_2|G_B)$. The group $A_5$ is simple. Therefore $D \simeq 1$ or $A_5$. These cases corresponds respectively to  cases $1$ and $2$ of the theorem.

It's need to check existence of the nonsingular curve $\bar{C} \subset \mathbb{F}_0$ for each of these cases. In the first case this curve obviously exists. Because we can choose binary forms $p_i \in \mathcal{I}^{St(\bar{A}_5)},\ i=0,1,2$ in \eqref{eqC} without multiple and common factors (see the generators of $\mathcal{I}^{St(\bar{A}_5)}$ in \eqref{grundform}). It remains to verify existence of the nonsingular curve $\bar{C}$ in the second case. Also we need to show that the parameter $d$ in \eqref{eqC} can be odd and even. This follows from the next lemma.

   \begin{lemma}
  \label{lem1}
  There exist  nonsingular curves $\bar{C} \in \mathbb{F}_0$ with odd and even parameter $d$ given by equation \eqref{eqC} and invariant under the diagonal action of group $St(A_5)$ on $\mathbb{F}_0 \simeq \mathbb{P}^1\times \mathbb{P}^1$.
    \end{lemma}
   \begin{proof}
     Consider the linear space of polynomials $\mathbb{C}[x,y]$. The space has the natural structure of $SL(2,\mathbb{C})$-module. Denote by $R_n \subset \mathbb{C}[x,y]$ the subspace of polynomials with degree $n$. We have $\Equat(\bar{C})\in R_2 \otimes R_{2d}$. It's known (see \cite[Exercise 11.11]{Fulton}) that $ R_2 \otimes R_{2d} \simeq R_{2d+2} \oplus R_{2d} \oplus R_{2d-2}$ as $SL(2,\mathbb{C})$-module.

  Consider a linear system $\mathcal{J}$ of $St(\bar{A}_5)$-invariant curves with bidegree $(2,2d)$ in $\mathbb{F}_0$.
   Obviously, we have $\mathcal{J} \simeq (R_2 \otimes R_{2d})^{St(\bar{A}_5)} \simeq R_{2d+2}^{St(\bar{A}_5)} \oplus R_{2d}^{St(\bar{A}_5)} \oplus R_{2d-2}^{St(\bar{A}_5)}$.

First, we prove existence of a nonsingular curve $\bar{C}$ with odd parameter $d$.
We take $d=30k +15, k \in \mathbb{N}, \ k\ge 2$. It's easy to check (see \eqref{grundform}) that each set $R_{2d+2}^{St(\bar{A}_5)}, R_{2d}^{St(\bar{A}_5)}$ and $R_{2d-2}^{St(\bar{A}_5)}$ is not empty. To apply Bertini theorem, we need to study the base points of  system $\mathcal{J}$.

   We have $(x_0t_1-x_1t_0)^2 R_{2d-2}(t_0,t_1)^{St(\bar{A}_5)} \in \mathcal{J}$. It's easy to check that $R_{2d-2}^{St(\bar{A}_5)}=\Phi_3^4 R_{60k-20}^{St(\bar{A}_5)}$, and $\Phi_2^2 R_{60(k-1)}^{St(\bar{A}_5)} \subset R_{60k-20}^{St(\bar{A}_5)}$ (see \eqref{grundform}). The space $R_{60(k-1)}^{St(\bar{A}_5)}$ has not a common factor, since $k \ge 2$.
Hence the base points of $\mathcal{J}$ lie in the union of sets: $x_0t_1-x_1t_0=0$, $\Phi_2(t_0,t_1)=0$ and $\Phi_3(t_0,t_1)=0$.

    Consider the projection $\xi: R_2 \otimes R_{2d} \simeq R_{2d+2} \oplus R_{2d} \oplus R_{2d-2} \rightarrow R_{2d+2}$. It is given by the polynomial $p_0t_0^2+2p_1t_0t_1+p_2t_1^2$. This polynomial defines intersection of the curve $\bar{C}$ and of diagonal $x_0t_1-x_1t_0=0$. We have $\Phi_2 \Phi_3 R_{60k}^{St(\bar{A}_5)} \subset R_{2d+2}^{St(\bar{A}_5)}$. The space $R_{60k}^{St(\bar{A}_5)}$ has not a common factor, since $k \ge 2$. Hence we can take a polynomial $\Equat(\bar{C}) \in R_2 \otimes R_{2d}$, such that $\xi(\Equat(\bar{C})) = \Phi_2 \Phi_3 h(t_0,t_1)$, where the forms $\Phi_2$, $\Phi_3$, $h(t_0,t_1)$ have not pairwise common factors.
Therefore the base points of $\mathcal{J}$ lie in the union of sets: $\Phi_2(t_0,t_1)=0$ and $\Phi_3(t_0,t_1)=0$.
However by choose of $\Equat(\bar{C})$ we get that the curve $\bar{C}$ is nonsingular at these sets.

It remains to prove existence of a nonsingular curve $\bar{C}$ with even parameter $d$. We take $d=30k, k \in \mathbb{N}, \ k\ge 3$. It is easy to check (see \eqref{grundform}) that each set $R_{2d+2}^{St(\bar{A}_5)}, R_{2d}^{St(\bar{A}_5)}$ and $R_{2d-2}^{St(\bar{A}_5)}$ is not empty. To apply Bertini theorem, we need to study the base points of  system $\mathcal{J}$.

We have $\Phi_1 \mathcal{J}'\subset \mathcal{J}$, where $\mathcal{J}'$ is a linear system of $St(\bar{A}_5)$-invariant curves with bidegree $(2,2d-30)$ in $\mathbb{F}_0$. By previous arguments, we know that the base points of $\mathcal{J}'$ lie in the union of sets: $\Phi_2(t_0,t_1)=0$ and $\Phi_3(t_0,t_1)=0$. Hence the base points of $\mathcal{J}$ lie in the union of sets: $\Phi_1(t_0,t_1)=0$, $\Phi_2(t_0,t_1)=0$ and $\Phi_3(t_0,t_1)=0$.

    Again consider projection $\xi$. We have $\Phi_1 \Phi_2 \Phi_3 R_{60(k-1)}^{St(\bar{A}_5)} \subset R_{2d+2}^{St(\bar{A}_5)}$. The space $R_{60(k-1)}^{St(\bar{A}_5)}$ has not a common factor, since $k \ge 3$. Hence we can take a polynomial $\Equat(\bar{C}) \in R_2 \otimes R_{2d}$, such that $\xi(\Equat(\bar{C})) = \Phi_1 \Phi_2 \Phi_3 h(t_0,t_1)$, where the forms $\Phi_1$, $\Phi_2$, $\Phi_3$, $h(t_0,t_1)$ have not pairwise common factors. Again the conditions of Bertini theorem are satisfied.
    \end{proof}

  It remains to describe the action of group $G$. The embedding of the group $G_B \simeq St(A_5)$ to $PGL(2,\mathbb{C}) \times PGL(2, \mathbb{C})$ defines a unique embedding $St(\bar{A}_5) \hookrightarrow SL(2,\mathbb{C}) \times SL(2, \mathbb{C})$. We note that $\Equat(\bar{C})$ is invariant under the action of $St(\bar{A}_5)$. Hence there is defined an action of $St(\bar{A}_5)$ on the surface $S$ given by equation \eqref{sur1}. The action of $St(\bar{A}_5)$ on coordinates $u_i,\ i=0,\ldots,d-1$ coincides with the action on monomials $x_0^i x_1^{d-1-i},\ i=0, \ldots, d-1$, respectively. The remaining arguments are obvious.
            \end{proof}

     The case $e > 0$ will be considered in the next theorem.
          \begin{theorem}
                                     \label{th2}
  Let $(S,G,\phi) \in \mathbb{CB}$ , and $G$ be a finite nonsolvable group. Suppose that $G_0\simeq 1$, $G_K \simeq \langle \iota \rangle \simeq \mathbb{Z}_2$, $\Sigma'=\Sigma$, and $S/\iota \simeq \mathbb{F}_n,\ n>0$.
Then there is a $G$-invariant curve $E$, which is the preimage of exceptional section $\mathbb{F}_e$. Consider the contraction of this curve $(S,G,\phi) \rightarrow (S',G,\phi')$, where a map $\phi': S' \dashrightarrow \mathbb{P}^1$ is defined by $\phi$.The surface $S'$ is given by the following equation in $\mathbb{P}(d+e,e,1,1)$ with the coordinates $(u:x:t_0:t_1)$:
    \begin{equation}
    \label{surf1}
             u^2+p_0(t_0,t_1)x^2+2p_1(t_0,t_1)x+p_2(t_0,t_1)=0,
         \end{equation}
   where $p_i, \ i=0,1,2$ are binary forms with degree $2d$, $2d+e$, $2d+2e$, respectively.
  The binary form $p_0p_2-p_1^2$ has no multiple factors.
    The map $\phi'$ is given by
  $$\phi':\ (u:x:t_0:t_1)\mapsto (t_0:t_1).    $$
Moreover, $e$ is even.
    The group $G$ is generated  by the maps:
   \begin{equation*}
     \begin{aligned}
       &u \mapsto -u,\\
          & (u:x:t_0:t_1) \mapsto (u:x+F_e(t_0,t_1):at_0+bt_1:ct_0+d't_1),
     \end{aligned}
   \end{equation*}
 where
             \begin{equation}
      \label{matr1}
                    \begin{pmatrix}
                           a & b\\
             c & d'
                 \end{pmatrix}
                      \in St(\bar{A}_5),
       \end{equation}
  and $F_e(t_0,t_1)$ is a some binary form with degree $e$, unique for each matrix \eqref{matr1}.
  The group $G$ is isomorphic to
     \begin{equation}
                 \label{opG}
                       G \simeq \left \{
                                 \begin{aligned}
                  &\bar{A}_5, \text{if $d$ is odd},\\
                         & \mathbb{Z}_2 \times A_5, \text{if $d$ is even}.
                                 \end{aligned}
                       \right.
    \end{equation}
 All possibilities for $G$ occur.
                                  \end{theorem}
         \begin{proof}
                  We will use the following construction to represent the surface $S$ by equations. Consider the morphism $\mathbb{F}_e \rightarrow \mathbb{P}(e,1,1)$, which is the blowing down of exceptional section $\mathbb{F}_e$. Introduce the coordinates $(x:t_0:t_1)$ on $\mathbb{P}(e,1,1)$. The morphism $\phi$ induces projection $(x:t_0:t_1)\mapsto (t_0:t_1)$. The curve $\bar{C}$ will be represented by the following equation in $\mathbb{P}(e,1,1)$:
       $$
          p_0(t_0,t_1)x^2+2p_1(t_0,t_1)x+p_2(t_0,t_1)=0.
              $$
   The form $\Disc(\bar{C})=p_0p_2-p_1^2$ has no multiple factors, since $\bar{C}$ is nonsingular.
                 We construct a double cover of $\mathbb{P}(e,1,1)$, branched along $\bar{C}$. We get the surface $S'$ given by the equations \eqref{surf1}.

          Note that $\degg(p_0)$ is even, since $\bar{C} \in 2 \Pic(\mathbb{F}_e)$. Denote the degree as $2d$.
                     The automorphism group of $\mathbb{P}(e,1,1)$ consists of the maps
     $$
                   (x:t_0:t_1) \mapsto (a'x+P_e(t_0,t_1):b't_0+c't_1:d't_0+v't_1).
           $$
        where $P_e$ is a binary form with degree $e$. We can choose coefficients $b',\ c',\ d',\ v'$, so that
                $$
        \begin{pmatrix}
                      b' & c'\\
              d'  & v'
          \end{pmatrix}
   \in SL(2,\mathbb{C}).
$$
     We have $G_B\simeq A_5 \not \subset \bar{A}_5$ (see \eqref{1}). Therefore $e$ is even. We conjugate $G_B$ to a group consisting of the following maps
     $$
     (x:t_0:t_1) \mapsto (v x+F_e(t_0,t_1): at_0+bt_1:ct_0+d't_1),
        $$
   where the coefficients $a$, $b$, $c$, $d'$ and the binary form $F_e$ satisfy conditions of the theorem. But $v=1$, since $A_5$ is a  simple group, and all it's characters $A_5 \rightarrow \mathbb{C}^*$ are trivial. Obviously, we get \eqref{opG}.

Finally, we need to prove that all possibilities for $G$ in  \eqref{opG} occur. It's sufficient to construct nonsingular curves $\bar{C}$   invariant under an action of $G_B$ with odd and even parameter $d$. We assume that $G_B$ is a group consisting of maps
$$
     (x:t_0:t_1) \mapsto (x: at_0+bt_1:ct_0+d't_1),
        $$
 with condition \eqref{matr1}.  Let $d$ is even. Then the curve $\bar{C}$ is represented by the following equation in $\mathbb{P}(4,1,1)$:
  $$\Phi_3(t_0,t_1) x^2+ \Phi_2(t_0,t_1) =0,$$
 where $\Phi_i$, $i=2$, $3$ are binary forms  in \eqref{grundform}.

Let's construct the curve $\bar{C}$ with odd $d$. Consider the equation of $\bar{C}$ in $\mathbb{P}(30,1,1)$:
 $$\Phi_1(t_0,t_1) x^2 +2 h(t_0,t_1)x + h'(t_0,t_1)=0,$$
 where $\Phi_1$ is a binary form in \eqref{grundform}, and $h$, $h'$ are some binary forms in $\mathcal{I}^{St(\bar{A}_5)}$. It's easy to check by counting of parameters that $h$ and $h'$ can be chosen, such that $\Disc(\bar{C})=\Phi_1 h'- h^2$ has no multiple factors. Then $\bar{C}$ is nonsingular.

           \end{proof}

       \subsection{Case, when $G_0\simeq 1$, $G_K \simeq \mathbb{Z}_2$, and $\Sigma' \ne \Sigma$.}
    \label{conbund2}
            Here we will apply arguments of \cite[Section 3.2]{Tsygankov-10}. Let $r= |\Sigma|$, and $m=|\Sigma'|$.

 Let $g_1: \widetilde{S} \rightarrow S$ be blowing up of the singular points of reducible fibres
$\Sigma \setminus \Sigma'$, and $g_2: \widetilde{S} \rightarrow S'$ be the contraction of  proper transform of $\Sigma \setminus \Sigma'$. The surface $S'$ has $2(r-m)$ singular points of type $A_1$. Obviously that maps  $g_1$ and $g_2$ are $G$-equivariant. We have the $G$-equivariant commutative diagram.
 \begin{equation}
 \label{diagramma}
 \xymatrix{
 & \widetilde{S} \ar@{->}[dl]_{g_1} \ar@{->}[dd]_{\widetilde{h}} \ar@{->}[dr]^{g_2}\\
 S \ar@{->}[dd]_h && S' \ar@{->}[dd]_{h'}\\
 & \widetilde{S}/\iota \ar@{->}[dl]_{g_1'} \ar@{->}[dr]^{g_2'}\\
 S/ \iota && S'/ \iota \\
 }
 \end{equation}

 In the diagram the vertical arrows correspond to the quotient map by the involution $ \iota $, and  maps $ g_1'$ and $ g_2' $ are induced by  maps $ g_1 $ and $ g_2 $. The triple $(S,G,\phi)$ defines a triple $(S',G,\phi')$, where the morphism $\phi': S' \rightarrow \mathbb{P}^1$ is induced by the morphism $\phi$.

   By \cite[Lemma 3.4]{Tsygankov-10} we get that surfaces $\widetilde{S}/\iota$ and $S'/ \iota$ are nonsingular. Moreover, $S'/\iota \simeq \mathbb{F}_e$.

  The morphism $h': S' \rightarrow \mathbb{F}_e$ is a double cover branched over the union of curves $C' \cup g_{2*}(g_1^*(\Sigma \setminus \Sigma'))$, where the curve $C'$ is the proper transform of  curve $C$. The image of the curve $g_{2*}(g_1^*(\Sigma \setminus \Sigma'))$ on the ruled surface $\mathbb{F}_e$ is the union of $r-m$ fibres. Denote these fibres as $S_i,\ i=1,\ldots , r-m$. Also let $\hat{C}=h'(C')$.

   For each fiber $S_i,\ i=1,\ldots , r-m$ denote by $x_{i1}$ and $x_{i2}$ two distinct points of the intersection $S_i \cap \hat{C}$. Obviously, there is defined a faithful action of $G_B$ (see exact sequence \eqref{1}) on $\mathbb{F}_e$. By \cite[Lemma 3.5]{Tsygankov-10} the triple $(S,G,\phi)$ is minimal, iff points $x_{i1}$ and $x_{i2}$ lie in the same orbit under the action of $G_B$ for each $i=1,\ldots , r-m$.

           We consider cases $e=0$ and $e>0$ in Theorems \ref{th3} and \ref{th4} respectively.

    We make some preparations before statement of Theorem \ref{th3}. Introduce the coordinates $(x_0:x_1,t_0:t_1)$ on $\mathbb{F}_0\simeq \mathbb{P}^1 \times \mathbb{P}^1$. The morphism $\phi': S' \rightarrow \mathbb{P}^1$ induces projection $\sigma: (x_0:x_1,t_0:t_1) \mapsto (t_0:t_1)$.
  The fibres $S_i \subset \mathbb{F}_0$, $i=1,\ldots , r-m$ are represented by the equations:
     $$
   S_i: a_i t_0 + b_i t_1=0.
    $$
  Consider the form $Q_{r-m}= \prod_{i=1}^{r-m} (a_i t_0 + b_i t_1) $. We conjugate the subgroup $G_B \simeq A_5 \subset PGL(2,\mathbb{C})$ to the group $St(A_5)$.  Let $\Stab_{St(A_5)}(\sigma(S_i))$, $i=1,\ldots,r-m$ be a stabilizer of  point $\sigma(S_i) \subset \mathbb{P}^1$ in the group $St(A_5)$. Considering equations \eqref{grundform}, we get that $\left| \Stab_{St(A_5)}(\sigma(S_i)) \right.|$ is either $1$, or $2$, or $3$, or $5$. Applying \cite[Lemma 3.5]{Tsygankov-10}, we get $\Stab_{St(A_5)}(\sigma(S_i))\simeq \mathbb{Z}_2$, $i=1, \ldots, r-m$. Hence $r-m=30$, and $Q_{r-m}=\Phi_1$.
The curve $\hat{C}$ is represented by the equation:
     \begin{equation}
\label{eqC1}
  \Equat(\hat{C})= p_0(t_0,t_1)x_0^2+2p_1(t_0,t_1)x_0x_1+p_2(t_0,t_1)x_1^2=0,
   \end{equation}
  where $p_i, \ i=0,1,2$ are binary forms with degree $2d$. Note that the degree is even, so as the divisor class $\hat{C}+ \sum_i S_i \in 2 \Pic{F}_0$ (see \cite[Lemma 3.6]{Tsygankov-10}), and $r-m=30$ is even. The form $\Disc(\hat{C})=p_0p_2-p_1^2$ has no multiple factors, since $\hat{C}$ is nonsingular.
We will apply the Segre embedding $\nu: \mathbb{P}^1\times \mathbb{P}^1 \rightarrow \mathbb{P}^3$ to represent the surface $S$ by equations.
Introduce the coordinates $(x:y:z:w)$ on $\mathbb{P}^3$. This embedding is given by $\nu: (x_0:x_1, t_0:t_1)\mapsto (x_0t_0:x_0t_1:x_1t_0:x_1t_1)$.
We choose some polynomials $F_i(x,y,z,w)$, $i=0,\ldots,2d+28$, such that $x_0^i x_1^{2d+28-i} \Phi_1 \Equat(\hat{C})=\nu^*(F_i)$. The surface $S$ is represented by the equations in $\mathbb{P}(d^{d+15},1^4)$ with the coordinates $u_i,\ x,\ y,\ z,\ w,\ i=0,\ldots,d+14$:
    \begin{equation}
  \label{sur3}
       \begin{aligned}
   u_i u_j=F_{i+j},\ 0 \le i \le j \le d+14,\\
           x^{j-i}u_i=u_j z^{j-i},\ y^{j-i}u_i=u_j w^{j-i},\ 0 \le i< j \le d+14,\\
            xw=yz.
 \end{aligned}
   \end{equation}

    \begin{theorem}
  \label{th3}
  Let $(S,G,\phi) \in \mathbb{CB}$ , and $G$ be a finite nonsolvable group. Suppose that $G_0\simeq 1$, $G_K \simeq \langle \iota \rangle \simeq \mathbb{Z}_2$, $\Sigma' \ne \Sigma$ and $S'/\iota \simeq \mathbb{F}_0$.
  Then there is defined a faithful action of $G_B$ (see \eqref{1}) on $\mathbb{F}_0$, and $G_B \subset PGL(2,\mathbb{C})\times PGL(2,\mathbb{C})$.  There is a $G$-invariant birational map $(S,G, \phi) \dashrightarrow (S', G,\phi')$ described in the diagram \eqref{diagramma}. The surface $S'$ can be represented by equations \eqref{sur3}. The parameter $d$ in \eqref{eqC1} is odd.
   The morphism $\phi': S' \rightarrow \mathbb{P}^1$ is given by
$$\phi':\ (u_0:\ldots:u_{d+14}:x:y:z:w) \mapsto \left \{
   \begin{aligned}
   &(x:y), \text{if}\ (x:y)\ne (0:0);\\
   &(z:w), \text{if}\ (z:w)\ne (0:0).
    \end{aligned}
    \right.              $$

   We have
   \begin{equation}
   \label{cond1}
   \Phi_1(t_0,t_1) \not | (p_0(t_0,t_1)t_0^2 + 2p_1(t_0,t_1)t_0t_1 + p_2(t_0,t_1)t_1^2),
   \end{equation}
    where $p_i,\ i=0,1,2$ are binary forms in \eqref{eqC1}, and $\Phi_1$ is the binary form from \eqref{grundform}.

  The group $G_B \subset PGL(2,\mathbb{C}) \times PGL(2, \mathbb{C})$ is the image of diagonal embedding $St(A_5)\hookrightarrow PGL(2,\mathbb{C}) \times PGL(2, \mathbb{C})$.
   The group $G$ is isomorphic to  $\bar{A}_5$.   This possibility for $G$ occur.

    The group $G$ acts on $S'$ by the following way. Embedding of $G_B \simeq St(A_5)$ to $PGL(2,\mathbb{C}) \times PGL(2, \mathbb{C})$ defines a unique embedding $St(\bar{A}_5) \hookrightarrow SL(2,\mathbb{C}) \times SL(2, \mathbb{C})$. This  defines an action of $St(\bar{A}_5)$ on the surface $S'$ given by the equation \eqref{sur3}. The action of $St(\bar{A}_5)$ on coordinates $u_i,\ i=0,\ldots,d+14$ coincides with the action on monomials $x_0^i x_1^{d+14-i},\ i=0, \ldots, d+14$.
An action of $G$ is generated by the action of $St(\bar{A}_5)$ and by the map
 $$(u_0:\ldots:u_{d+14}:x:y:z:w) \rightarrow (-u_0:\ldots:-u_{d+14}:x:y:z:w).$$
 \end{theorem}

      \begin{proof}
  The subgroup $G_B \subset \Aut(\mathbb{F}_0)$ acts trivially on $\Pic(\mathbb{F}_0)$. Hence $G_B \subset PGL(2,\mathbb{C}) \times PGL(2,\mathbb{C})$. Consider projections $\pi_i: PGL(2,\mathbb{C}) \times PGL(2,\mathbb{C}) \rightarrow PGL(2,\mathbb{C})$, $i=1,2$ on the first and the second factor respectively.
       By Goursat's Lemma (see \cite[Lemma 4.1]{Dolgachev-Iskovskikh}) we get that $G_B \simeq \pi_1(G_B) \triangle_D \pi_2(G_B)$, where $D \simeq \Imm(\pi_1|G_B)/\Ker(\pi_2|G_B)$. We have $D \simeq 1$ or $A_5$.

      By \cite[Lemma 3.5]{Tsygankov-10} we need to find conditions, when points $x_{i1}$ and $x_{i_2}$ for each $i=1,\ldots,30$ lie in the same orbit under an action of $G_B$. Obviously, $D\simeq A_5$. We conjugate $G_B \subset PGL(2,\mathbb{C})\times PGL(2,\mathbb{C})$ to the image of diagonal embedding $St(A_5) \hookrightarrow PGL(2,\mathbb{C})\times PGL(2,\mathbb{C})$. It's easy to check that none of points $x_{i1}$ and $x_{i_2}$ for each $i=1,\ldots,30$ lies on  the diagonal $x_0t_1-x_1t_0=0$. This is a sufficient condition, and it's equivalent to \eqref{cond1}.

    Therefore we need to prove existence of  curves $\hat{C} \subset \mathbb{F}_0$ with odd  parameter $d$, such that  condition \eqref{cond1} holds. Also we will prove that $d$ cannot be even.
       We will use notations and arguments of Lemma \ref{lem1}.
      Consider a linear system $\mathcal{J}$ of $St(\bar{A}_5)$-invariant curves with bidegree $(2,2d)$ in $\mathbb{F}_0$. Consider the projection $\xi: R_2 \otimes R_{2d} \rightarrow R_{2d+2}$. It's is given by the polynomial $p_0t_0^2+2p_1t_0t_1+p_2t_1^2$.

Suppose that $d$ is even. It's easy to check that any polynomial $f(t_0,t_1)\in R_{2d+2}^{St(\bar{A}_5)}$  is  divided by $\Phi_1$. Therefore it's impossible.

Let $d$ is odd. In Lemma \ref{lem1} we proved that a general member of $\mathcal{J}$ is nonsingular, if $d=30k + 15, k \in \mathbb{N}, \ k\ge 2$. But it's obvious that there exist polynomial $f(t_0,t_1)\in R_{2d+2}^{St(\bar{A}_5)}$ with degree $2d+2=60k+32$, $k\ge 2$, which is not divided by $\Phi_1$.

  The remaining arguments follow from the construction of  equations \eqref{sur3} and are obvious.
     \end{proof}

In the next theorem we consider case $e>0$.
    \begin{theorem}
   \label{th4}
                     Let $(S,G,\phi) \in \mathbb{CB}$, and $G$ be a finite nonsolvable group. Suppose that $G_0\simeq 1$, $G_K \simeq \langle \iota \rangle \simeq \mathbb{Z}_2$, $\Sigma' \ne \Sigma$, and $S'/\iota \simeq \mathbb{F}_e$, $e>0$.    Then there is a $G$-invariant birational map $(S,G, \phi) \dashrightarrow (S', G,\phi')$ described in diagram \eqref{diagramma}. The surface $S'$ contains a $G$-invariant curve $E$, which is the preimage of exceptional section $\mathbb{F}_e$.
 Consider the contraction of this curve $(S',G,\phi') \rightarrow (S'',G,\phi'')$, where a map $\phi'': S'' \dashrightarrow \mathbb{P}^1$ is defined by $\phi'$.The surface $S''$ is given by the following equation in $\mathbb{P}(d+e+15,e,1,1)$ with the coordinates $(u:x:t_0:t_1)$:
              \begin{equation}
 \label{surf2}
             u^2+\Phi_1(t_0,t_1) (p_0(t_0,t_1)x^2+2p_1(t_0,t_1)x+p_2(t_0,t_1))=0,
         \end{equation}
   where $p_i, \ i=0,1,2$ are binary forms with degree $2d$, $2d+e$, $2d+2e$, respectively, and $\Phi_1$ is the binary form from \eqref{grundform}. Also $\Phi_1 \not | (p_0 p_2-p_1^2)$.

The map $\phi''$ is given by
$$\phi'':\ (u:x:t_0:t_1) \mapsto (t_0:t_1).$$

  Moreover, $e \equiv 2$ $(mod$ $4)$.
   The group $G$ is generated  by the maps:
   \begin{equation*}
     \begin{aligned}
       &u \mapsto -u,\\
          & (u:x:t_0:t_1) \mapsto (u:x+F_e(t_0,t_1):at_0+bt_1:ct_0+d't_1),
     \end{aligned}
   \end{equation*}
 where
             \begin{equation}
      \label{matr2}
                    \begin{pmatrix}
                           a & b\\
             c & d'
                 \end{pmatrix}
                      \in St(\bar{A}_5),
       \end{equation}
  and $F_e(t_0,t_1)$ is a some binary form with degree $e$, unique for each matrix \eqref{matr2}.
  The group $G$ is isomorphic to
     \begin{equation}
          \label{opG2}
                       G \simeq \left \{
                                 \begin{aligned}
                  &\bar{A}_5, \text{if $d$ is even},\\
                         & \mathbb{Z}_2 \times A_5, \text{if $d$ is odd}.
                                 \end{aligned}
                       \right.
    \end{equation}
                   All possibilities for $G$ occur.
                                  \end{theorem}
         \begin{proof}
                  We will use the following construction to represent the surface $S''$ by equations. Consider the morphism $\mathbb{F}_e \rightarrow \mathbb{P}(e,1,1)$, which is the blowing down of  exceptional section $\mathbb{F}_e$.          Introduce on $\mathbb{P}(e,1,1)$ the coordinates $(x:t_0:t_1)$. The map $\phi':S' \rightarrow \mathbb{P}^1$ induces the projection $\sigma: (x:t_0:t_1) \mapsto (t_0:t_1)$. The fibres $S_i \subset \mathbb{F}_e$, $i=1,\ldots,r-m$ are represented by the equations:
     $$
   S_i: a_i t_0 + b_i t_1=0.
    $$
  Consider the form $Q_{r-m}= \prod_{i=1}^{r-m} (a_i t_0 + b_i t_1) $. We conjugate the group $G_B \simeq A_5 \subset PGL(2,\mathbb{C})$ to the group $St(A_5)$. Let $\Stab_{St(A_5)}(\sigma(S_i))$, $i=1,\ldots,r-m$ be a stabilizer of  point $\sigma(S_i) \subset \mathbb{P}^1$ in the group $St(A_5)$.  Considering equations \eqref{grundform}, we get that $\left| \Stab_{St(A_5)}(\sigma(S_i)) \right.|$ is either $1$, or $2$, or $3$, or $5$. Applying \cite[Lemma 3.5]{Tsygankov-10}, we get $\Stab_{St(A_5)}(\sigma(S_i))\simeq \mathbb{Z}_2$, $i=1, \ldots, r-m$. Hence $r-m=30$, and $Q_{r-m}=\Phi_1$ (see \eqref{grundform}).
The curve $\hat{C}$ is represented by the following equation in $\mathbb{P}(e,1,1)$ with the coordinates $(x:t_0:t_1)$:
       $$
          p_0(t_0,t_1)x^2+2p_1(t_0,t_1)x+p_2(t_0,t_1)=0.
              $$
   Each fibre $S_i$, $i=1, \ldots,30$ intersects the curve $\hat{C}$ in two distinct points: $x_{i1}$ and $x_{i2}$. Hence $\Phi_1 \not | (p_0 p_2-p_1^2)$.
                 We construct a double cover of $\mathbb{P}(e,1,1)$ branched along $\hat{C}$ and $\Phi_1(t_0,t_1)=0$. We get the surface $S''$ given by the equation \eqref{surf2}.

          Note that $\degg(p_0)$ is even, since $\hat{C} + \sum_i S_i\in 2 \Pic(\mathbb{F}_e)$ (see \cite[Lemma 3.6]{Tsygankov-10}), and $r-m=30$ is even. Denote the degree as $2d$.

Then we apply the arguments as in Theorem \ref{th2}. We prove that $e$ is even, and conjugate $G_B$ to a group consisting of the following maps
     $$
     (x:t_0:t_1) \mapsto (x+F_e(t_0,t_1): at_0+bt_1:ct_0+d't_1),
                 $$
 where the coefficients $a$, $b$, $c$, $d'$ and the binary form $F_e$ satisfy conditions of the Theorem.

    By \cite[Lemma 3.5]{Tsygankov-10} we need to find conditions, when the points $x_{i1}$ and $x_{i_2}$ for each $i=1,\ldots,30$ lie in the same orbit under an action of $G_B$. Obviously, it's sufficient to check, that each element $g \in G_B$ with $\ord(g)=2$ doesn't have fixed points on the curve $\hat{C}$. The element $g$ can be conjugated to the map:
 $$
   (x:t_0:t_1) \mapsto (x:it_0: -it_1).
     $$
 It's easy to see that $g$ doesn't have fixed points on $\hat{C}$, iff $e \equiv 2$ $(mod$ $4)$. We easily get \eqref{opG2}.

   We need to prove existence of curves $\hat{C} \subset \mathbb{F}_e$ with odd and even parameters $d$, such that listed above conditions holds. We assume that $G_B$ is a group consisting of maps
$$
     (x:t_0:t_1) \mapsto (x: at_0+bt_1:ct_0+d't_1),
        $$
 with condition \eqref{matr2}.

Let $d$ is even. We take $e=34$, and the curve $\hat{C}$ is represented by the following equation in $\mathbb{P}(34,1,1)$ with the coordinates $(x:t_0:t_1)$
   $$
               \Phi_3(t_0,t_1)x^2+Q_{60}(t_0,t_1)\Phi_2(t_0,t_1)=0,
    $$
  where $\Phi_2$ and $\Phi_3$ are binary forms from \eqref{grundform}, and $Q_{60}\in \mathcal{I}^{St(\bar{A}_5)}$ (see Notation \ref{mathcalI}) is a some binary form with degree $60$ and without multiple factors.

Let $d$ is odd. We can employ here example constructed in proof of Theorem \ref{th2}.

     \end{proof}

    \subsection{Case, when $G_0\simeq 1$, $G_K \simeq \mathbb{Z}_2^2$.}
\label{conbund3}
   In this section we prove the next theorem.
 \begin{theorem}
  \label{th5}
  Let $(S,G,\phi) \in \mathbb{CB}$, and $G$ be a finite nonsolvable group. Suppose that $G_0\simeq 1$, $G_K \simeq  \mathbb{Z}_2^2$. Then there exists an embedding $S \hookrightarrow \mathbb{P}(\mathcal{E})$, where $\mathcal{E}$ is a line bundle on $\mathbb{P}^1$. We have $\mathcal{E}= \mathcal{E}_0 \oplus \mathcal{E}_1 \oplus \mathcal{E}_2$, and isomorphisms $f_i: \mathcal{E}_i \rightarrow \mathcal{O}(a_i)$, $i=0,1,2$, $a_0=0$, $0\le a_1 \le a_2$.
The surface $S$ can be represented by the equation in $\mathbb{P}(\mathcal{E}) \simeq  \mathbb{P}(\mathcal{O}\oplus \mathcal{O}(a_1) \oplus \mathcal{O}(a_2))$:
\begin{equation}
    \label{eq3}
     \sum_{i=0}^2 \sum_{j,k=0}^{a_i} p_i^{j,k}(t_0,t_1) \xi_i^j \xi_i^k=0,
             \end{equation}
  where $p_i^{j,k}$ are binary forms with degree $d$, and $\xi_i^j=f_i^{-1}(t_0^j t_1^{a_i-j})$, $i=0,1,2$, $0\le j \le a_i$. The morphism $\phi: S \rightarrow \mathbb{P}^1$ is induced by the natural projection $\mathbb{P}(\mathcal{E}) \rightarrow \mathbb{P}^1$.
   The following conditions holds.
\begin{equation}
 \label{cond3}
                      \begin{aligned}
          & H_0=p_0^{0,0}(t_0,t_1) \in \mathcal{I}^{St(\bar{A}_5)},\\
      &H_1=\sum_{j,k=0}^{a_1} p_1^{j,k}(t_0,t_1) t_0^{j+k}t_1^{2a_1-j-k} \in \mathcal{I}^{St(\bar{A}_5)},\\
    & H_2=\sum_{j,k=0}^{a_2} p_2^{j,k}(t_0,t_1) t_0^{j+k}t_1^{2a_2-j-k} \in \mathcal{I}^{St(\bar{A}_5)}.
              \end{aligned}
\end{equation}
Also the binary forms $H_0$, $H_1$ and $H_2$ do not have multiple and pairwise common factors.

 The group $G_K$ acts by the following way (see Theorem \ref{th1.4}).
\begin{equation*}
    \begin{aligned}
  \iota_0(\xi)=\mp \xi_0 \pm \xi_1 \pm \xi_2,\\
           \iota_1(\xi)=\pm \xi_0 \mp \xi_1 \pm \xi_2,\\
      \iota_2(\xi)=\pm \xi_0 \pm \xi_1 \mp \xi_2,\\
 \end{aligned}
 \end{equation*}
    for any $\xi=\xi_0+\xi_1+\xi_2$, $\xi_i \in \mathcal{E}_i$, $i=0,1,2$.

   The action of $G$ on the surface $S$ is generated by the action of $G_K$ and an action of $St(\bar{A}_5)$.
 The action of $St(\bar{A}_5)$ on sections $\xi_i^j=f_i^{-1}(t_0^j t_1^{a_i-j})$, $i=0,1,2$, $0\le j \le a_i$ is induced by the action on $\mathbb{C}[t_0,t_1]$. We have
    \begin{equation}
   \label{G}
               G \simeq \left \{
           \begin{aligned}
          & \mathbb{Z}_2 \times \bar{A}_5, \text{if either $a_1$, or $a_2$ is odd},\\
          & \mathbb{Z}_2^2 \times A_5, \text{otherwise.}
       \end{aligned}
        \right.
     \end{equation}
All possibilities for $G$ occur.
     \end{theorem}
  \begin{proof}
      Denote as $f$ the fibre's divisor class of the conic bundle $(S,G,\phi)$. We have $-K_S \cdot f=2$.
      It's well known that a line bundle $\mathcal{O}(-K_S)$ is locally free of rank $3$.   Hence the line bundle $\mathcal{O}(-K_S)$ is relatively very ample and defines an embedding $S \hookrightarrow \mathbb{P}(\mathcal{E}')$, where $\mathcal{E}'=\phi_*(\mathcal{O}(-K_S))$. By Grothendieck theorem we have $\mathcal{E}'=\mathcal{O}(a_0')\oplus \mathcal{O}(a_1')\oplus \mathcal{O}(a_2')$. Obviously, we  can take $a_0' \le a_1' \le a_2'$. Hence we can take the bundle $\mathcal{E}$ in the statement of theorem to be equal $\mathcal{E}=\mathcal{E}' \otimes \mathcal{O}(-a_0')$.

  An action of $G$ on $\mathcal{O}(-K_S)$ defines an action on $\mathcal{E}$. In the next lemma we show
  that the action of $G_K$ on $\mathcal{E}$ is diagonalizable.
         \begin{lemma}
         \label{lem2}
  We can choose a  decomposition $\mathcal{E} = \mathcal{E}_0\oplus \mathcal{E}_1\oplus \mathcal{E}_2$, where $\mathcal{E}_i \simeq \mathcal{O}(a_i)$, $i=0,1,2$, such that $G_K$ acts by the following way. Denote three different nontrivial elements in $G_K \simeq \mathbb{Z}_2^2$ as $\iota_0$, $\iota_1$, $\iota_2$.
Then
     \begin{equation}
    \begin{aligned}
           \label{generatK}
  \iota_0(\xi)=\mp \xi_0 \pm \xi_1 \pm \xi_2,\\
           \iota_1(\xi)=\pm \xi_0 \mp \xi_1 \pm \xi_2,\\
      \iota_2(\xi)=\pm \xi_0 \pm \xi_1 \mp \xi_2,\\
 \end{aligned}
 \end{equation}
    for any $\xi=\xi_0+\xi_1+\xi_2$, $\xi_i \in \mathcal{E}_i$, $i=0,1,2$.
    \end{lemma}
   \begin{proof}
 Recall that $a_0=0$, $0 \le a_1 \le a_2$. Also remind that $G_K$ acts trivially on the base of fibration $\phi$.

Suppose that $0 < a_1 < a_2$. Then each bundle $\mathcal{E}_i$, $i=0,1,2$ is invariant under the action of $G_K$. Hence the statement is obvious.

Suppose that $a_i=a_j$, $a_i\ne a_k$ for some $i\ne j \ne k \ne i$. Without loss of generality we can take
 $0=a_1 < a_2$.  Then the action  of $G_K$ on $\mathcal{E}_0 \oplus \mathcal{E}_1$ defines an embedding $G_K \hookrightarrow GL(2,\mathbb{C})$. But, obviously, any subgroup $\mathbb{Z}_2^2 \subset GL(2,\mathbb{C})$ is diagonalizable.

 Suppose that $0=a_1=a_2$. Then the statement follows from the fact that any subgroup $\mathbb{Z}_2^2 \subset GL(3,\mathbb{C})$ is diagonalizable.
      \end{proof}

  We apply Lemma \ref{lem2}.  Fix isomorphisms $f_i: \mathcal{E}_i \rightarrow \mathcal{O}(a_i)$, $i=0,1,2$. Let $\xi_i^j=f_i^{-1}(t_0^j t_1^{a_i-j})$, $i=0,1,2$, $0\le j \le a_i$ be  generators of global section spaces of  bundles $\mathcal{E}_i$.
     Then the surface $S$ is presented by the following equation in $\mathbb{P}(\mathcal{E})$
   $$
     \sum_{0 \le i\le j \le 2} \sum_{k\le a_i, l\le a_j} p_{i,j}^{k,l}(t_0,t_1)  \xi_i^k \xi_j^l =0,
       $$
    where $p_{i,j}^{k,l}(t_0,t_1)$ are binary forms with degree $d$. But it easily follows from equations \eqref{generatK} that $p_{i,j}^{k,l}=0$, if $i\ne j$. Hence the surface $S$ can be represented by equation
           \eqref{eq3}.

    Let's find relations on the forms $p_i^j$, $0\le j \le a_i$, $i=0,1,2$.
 By Theorem \ref{th1.4} each nontrivial element $\iota_i,\ i=0,1,2$ of the subgroup $G_K$ fixes pointwise
   an irreducible smooth bisection $C_i$. Hence, there is defined an action of $G$ on the set of these curves, since $G_K\triangleleft G$. This action defines a homomorphism $\sigma: G \rightarrow S_3$. But $G_B\simeq A_5$ is simple.
Therefore $\sigma$ is trivial. The curves $C_i$, $i=0,1,2$ on the surface $S$ are cut out by the hypersurfaces:
  $$
     \xi_i^j=0,\ 0\le j \le a_i.
   $$
We conjugate $G_B \subset PGL(2,\mathbb{C})$ to $St(A_5)$. We employ now notations \eqref{cond3}. From triviality of $\sigma$ we get: $H_0 H_1$, $H_0 H_2$, $H_1 H_2 \in \mathcal{I}^{St(\bar{A}_5)}$. Hence $H_0^2$, $H_1^2$, $H_2^2 \in \mathcal{I}^{St(\bar{A}_5)}$. One knows that all characters of $St(\bar{A}_5)$ are trivial (see \eqref{grundform}). Therefore we get \eqref{cond3}.

It's easy to check that the surface $S$ is nonsingular, iff the binary forms $H_0$, $H_1$ and $H_2$
do not have multiple and pairwise common factors.

Now we can describe an action of group $G$ on the surface $S$. The  action of $St(\bar{A}_5)$ on $\mathbb{C}[t_0,t_1]$ induces an action on $\xi_i^j=f_i^{-1}(t_0^j t_1^{a_i-j})$, $i=0,1,2$, $0\le j \le a_i$.
   The action of $G$  is generated by the action of $G_K$ and the action of $St(\bar{A}_5)$.
 Therefore we easily get \eqref{G}. It's easy to check that all possibilities for $G$ in \eqref{G} occur.
       \end{proof}

                 \section{Conjugacy question.}
                      \label{Conjugate}
    The main result of this section is Theorem \ref{th6}. Here we reprove results of \cite[Section 8]{Dolgachev-Iskovskikh} for the sake of completeness.
     As the main tool we will use the next theorem.

            \begin{theorem}[{\cite[Theorem 1.6]{Iskovskikh-96}}]
   \label{thIsk}
        \begin{enumerate}
         \item Let $(S,G)$ be a surface in the class $\mathbb D$ with degree $K_S^2=1$, and let
                $\chi : S \dashrightarrow S'$ be a birational $G$-invariant map onto an arbitrary surface $(S',G) \in \mathbb D \cup  \mathbb {CB} $.
Then $S'$, like $S$, is a del Pezzo surface of degree $1$ and $\chi$ is an isomorphism.
                   \item Let $\chi: S \dashrightarrow   S'$ be a birational map, where $(S,G) \in  \mathbb D$ and $(S',G)     \in  \mathbb D  \cup  \mathbb{CB} $.
Suppose that $S$ has no  points $x$   with $|\Orb_G(x)| < K_S^2$, where $\Orb_G(x)$ is an orbit of  point $x$ under action of  $G$. Then $\chi$ is an
isomorphism.
                  \item Let $\chi : S \dashrightarrow S'$ be a birational $G$-invariant map, where $(S,G,\phi) \in  \mathbb {CB} $ and $(S',G) \in  \mathbb D  \cup  \mathbb {CB}$.
Suppose that $K_S^2\le 0$; then $(S',G,\phi') \in  \mathbb {CB}$, $K_S^2 = K_{S'}^2$, and $\chi$ takes a pencil
of conics on $S$ to a pencil of conics on $S'$, that is, the diagram
               $$
   \xymatrix{
    S \ar@{-->}[r]^\chi \ar[d]_\phi & S'\ar[d]_{\phi'}\\
     \mathbb{P}^1 \ar[r]^\pi   & \mathbb{P}^1
}
                                               $$
       is commutative, where $\pi$ is an isomorphism over $\mathbb{C}$.
    \end{enumerate}
            \end{theorem}

                \begin{theorem}
             \label{th6}
                     \begin{enumerate}
          \item \label{1case} Let $(S,G,\phi)$ be a surface in the class $\mathbb{CB}$, $K_S^2\le 0$, and $G$ be a finite nonsolvable group. Let $\chi: S \dashrightarrow S'$ be a birational $G$-invariant map, where $(S',G) \in \mathbb{D} \cup \mathbb{CB}$. Then $(S',G,\phi') \in \mathbb{CB}$, and $K_S^2=K_{S'}^2$. The map $\chi$ is a composition of elementary transformations $\elm_{x_1}\circ \ldots \circ \elm_{x_n}$, where $(x_1,\ldots,  x_n)$ is a $G$-invariant set of points not lying
on a singular fibre with no two points lying on the same fibre.
           \item Let $(S,G,\phi)$ be a surface in the class $\mathbb{CB}$, $K_S^2> 0$, and $G$ be a finite nonsolvable group.     Let $\chi: S \dashrightarrow S'$ be a birational $G$-invariant map, where $(S',G)\in \mathbb{D} \cup \mathbb{CB}$.
Then $K_S^2=8$, $S\simeq \mathbb{F}_n$, $n\ne 1$.
        \begin{enumerate}
       \item \label{case1} Let $n$ is odd. Then one the following cases occurs
           \begin{equation}
       \begin{aligned}
      & S'\simeq \mathbb{P}^2, \text{and}\ G\simeq \mathbb{Z}_{n'} \times \bar{A}_5;\\
     & S' \simeq \mathbb{F}_m, \text{where}\ m\ \text{is odd and}\ m\ne 1.
             \end{aligned}
           \end{equation}
          \item
\label{case2}       Let $n$ is even. Then $S'$ can be isomorphic to $\mathbb{F}_m$, where $m$ is even. If $n=0$, and $m\ne 0$, then $G\simeq \mathbb{Z}_m \times A_5$.
           \end{enumerate}
\item   Let $(S,G)$ be a surface in the class $\mathbb{D}$, and $G$ be a finite nonsolvable group. Let $\chi: S \dashrightarrow S'$ be a birational $G$-invariant map, where $(S',G)\in \mathbb{D} \cup \mathbb{CB}$. Then we have the following.
          \begin{enumerate}
  \item   Let $S\simeq \mathbb{P}^2$, and $G\simeq \mathbb{Z}_m \times \bar{A}_5$. Then the surface $S'$ may be isomorphic to either $\mathbb{P}^2$, or $\mathbb{F}_n$, where $n$ is odd.
     \item Otherwise we have $S' \simeq S$. If $K_S^2 < 9$, then $\chi$ is an automorphism of $S$.
                 \end{enumerate}
    \end{enumerate}
                 \end{theorem}
    \begin{remark}
I.Cheltsov proved the following. If $(\mathbb{F}_n,G,\phi)\in \mathbb{CB}$, $n\ne 1$, and $G$ be a finite nonsolvable group, then there exist a birational $G$-invariant map $\chi: \mathbb{F}_n \dashrightarrow S'$, where
     \begin{equation*}
       \begin{aligned}
      & S'\simeq \mathbb{P}^2, \text{if $n$ is odd};\\
     & S' \simeq \mathbb{F}_0, \text{if $n$ is even}.
             \end{aligned}
           \end{equation*}
   \end{remark}
        \begin{proof}
  The first case of  theorem follows from Theorem \ref{thIsk} and \cite[Theorems 7.7, Proposition 7.14]{Dolgachev-Iskovskikh}.

Let's prove the second case of  theorem. First we prove that $K_S^2=8$. Denote by $r$ the number of the reducible fibres of  conic bundle $(S,G,\phi)$. Suppose that $r\ne 0$. Obviously, we have $G_B \simeq A_5$ (see exact sequence \eqref{1} and Klein's classification of the finite subgroups in $PGL(2,\mathbb{C})$ in \cite[Section 5.5]{Dolgachev-Iskovskikh}). Then by \eqref{grundform} we have $r \ge 12$. Hence by Noether formula $K_S^2=8-r \le -4$. Therefore $r=0$, $K_S^2=8$.

  We will use results of Theorem \ref{th01} and theory of elementary links (see \cite[Section 7.2]{Dolgachev-Iskovskikh}). By \cite[Theorem 7.7]{Dolgachev-Iskovskikh} the map $\chi$ is equal to a composition of elementary links $\chi_1 \circ \ldots \circ \chi_k$. It's easy to check that $\chi_k$ is an elementary link of type II (see Theorem \ref{th01}). We have $\chi_k(\mathbb{F}_n)\simeq \mathbb{F}_l$. Then we apply Theorem \ref{th01}. For any point $x \in \mathbb{P}^1$ we have $|\Orb_{St(A_5)}(x)|$ is even (see Notation \ref{stA5}). Hence we easily get that $l-n$ is even.

Therefore if $n$ is even, then $\chi_i$, $i=1,\ldots,k$ are elementary links of type II. And we easily get the case \ref{case2} of  theorem.

Consider case, when $n$ is odd.  Then  $G \simeq \mathbb{Z}_{n'}\times \bar{A}_5$ by Theorem \ref{th01}.  Suppose that $\chi$ is not a composition of elementary links of type II. Then one of elementary links $\chi_i$, $1 \le i \le k-1$ must be a link of type III.
We may suppose that the links $\chi_j$, $i<j \le k$ are of type II.
In this case $\chi_i\circ \ldots \circ \chi_k(S)\simeq \mathbb{P}^2$. Below we will see that $S'\not \simeq X$, where $(X,G)\in \mathbb{D}$, and $K_X^2<9$. Therefore we get the case \ref{case1} of theorem.

  Let's prove the third case of theorem. To apply Theorem \ref{thIsk}, we need to show, that $S$ has no  points $x$   with $|\Orb_G(x)| < K_S^2$. We will argue, considering different values of $K_S^2$.

 If  $K_S^2$ is equal to either $7$, or $6$, or $4$ , or $1$, then by Theorem \ref{thdp3} there is no such pairs $(S,G)$.

    Let $K_S^2=2$. We apply Theorem \ref{thdp1}. Consider the following two automorphisms of surface $S$, given by equation \eqref{K_S2}:
   \begin{equation*}
    \begin{aligned}
    & \alpha: (T_0:T_1:T_2:T_3) \mapsto (T_1:T_2:T_0:T_3),\\
     & \beta: (T_0:T_1:T_2:T_3) \mapsto (\varepsilon_7 T_0: \varepsilon_7^4 T_1: \varepsilon_7^2 T_2: T_3).
    \end{aligned}
    \end{equation*}
It's easy to check by calculations that $\alpha$, $\beta\in G$, and there is no point $x\in S$ fixed under an action of  subgroup $H\subset G$ generated by $\alpha$ and $\beta$.

    Let $K_S^2=3$. We apply Theorem \ref{thdp2}. Consider the maps \eqref{K_S3eq}. Again by easy calculations we can check that there is no point $x\in S$, such that $|\Orb_G(x)|< 3$.

    Let $K_S^2=5$. We apply Theorem \ref{thdp4}. It's sufficient to consider the case $G\simeq A_5$. Suppose that there is a point $x\in S$, such that $|\Orb_G(x)|<K_S^2=5$. Denote by $\Stab_G(x)$ the stabilizer of point $x$ in the group $G$. Then $\Stab_G(x)$ is a subgroup of $G$ with order either $15$, or $20$, or $30$ or $60$.  It's well known that there are no subgroups of $A_5$ with order either $15$, or $20$, or $30$. Hence $|\Stab_G(x)|=60=|G|$. But it's easy to see from \eqref{K_S5eq} that the action of group $G$ on the surface $S$ has no fixed points. Contradiction.

   Let $K_S^2=8$. We apply Theorem \ref{K_S8}. Consider the action of $St(A_5)$ on $\mathbb{P}^1$. It's known (see \eqref{grundform}) that $|\Orb_{St(A_5)}(x)|\ge 12$ for any point $x \in \mathbb{P}^1$. Therefore $|\Orb_G(x)|\ge 12$ for any point $x \in S$.

  The case $K_S^2=9$ easily follows from the above investigation and Theorem \ref{K_S9}. However in this case the
 condition $|\Orb_{St(A_5)}(x)| \ge 9$ for any point $x \in S\simeq \mathbb{P}^2$ not always holds. For example in the second case of Theorem \ref{K_S9}. Therefore we cannot apply Theorem \ref{thIsk} to prove that $\chi$ is an isomorphism. But in fact we need only to know an isomorphism class of $S'$ to describe conjugacy classes.
    \end{proof}




\begin{corollary}
Suppose that the conditions of  case \ref{1case} of Theorem \ref{th6} are satisfied. After suppose that $G_K
\not \simeq \mathbb{Z}_2$ (see \eqref{1}). Then $\chi$ is an isomorphism.
\end{corollary}
\begin{proof}
By Theorems \ref{th1.4} and \ref{thExcept} we have $G_K \simeq \mathbb{Z}_2^2$ or $D_n$, $n \ge 2$. Therefore it's easy to see that any $G$-orbit of points $(x_1,\ldots,x_n)$ has two points lying on the same fibre of conic bundle $\phi$. Contradiction to  case \ref{1case} of Theorem \ref{th6}.
\end{proof}

   \section{The list of finite nonsolvable subgroups in $\Cr_2(\mathbb{C})$.}
     \label{List}
   Summarizing results obtained in sections \ref{DelPezzo} and \ref{ConBundle} we get the following list of finite nonsolvable subgroups in $\Cr_2(\mathbb{C})$.

   \begin{itemize}
   \item   $L_2(7)$

     This group is presented in $\Cr_2(\mathbb{C})$ by pairs $(S,L_2(7)) \in \mathbb{D}$, which were described in  Theorems \ref{K_S9} and \ref{thdp1}.
    \item $\mathbb{Z}_2 \times L_2(7)$

  This group is presented in $\Cr_2(\mathbb{C})$ by pairs $(S,\mathbb{Z}_2 \times L_2(7)) \in \mathbb{D}$, which were described in   Theorem \ref{thdp1}.
   \item $A_6$

   This group is presented in $\Cr_2(\mathbb{C})$ by pairs $(S,A_6) \in \mathbb{D}$, which were described in  Theorem \ref{K_S9}.
   \item $S_5$

    This group is presented in $\Cr_2(\mathbb{C})$ by pairs $(S,S_5) \in \mathbb{D}$, which were described in  Theorems \ref{thdp4}, \ref{thdp2}.
    \item $St(A_5)\wr \langle \tau \rangle$

       This group is presented in $\Cr_2(\mathbb{C})$ by pairs $(S,St(A_5)\wr \langle \tau \rangle) \in \mathbb{D}$, which were described in Theorem \ref{K_S8}.
          \item  $A_5 \times A_5$, $A_5 \times S_4$, $A_5 \times A_4$.

   These groups are presented in $\Cr_2(\mathbb{C})$ by triples $(S,G,\phi) \in \mathbb{CB}$ ($G$ is one of our groups), which were described in  Theorem \ref{th01}.
    \item $D_n \times \bar{A}_5$, $n\ge 3$

  This group is presented in $\Cr_2(\mathbb{C})$ by triples $(S,D_n \times \bar{A}_5, \phi) \in \mathbb{CB}$, which were described in  Theorem \ref{thExcept}.
   \item $D_n \times A_5$, $n\ge 3   $

   This group is presented in $\Cr_2(\mathbb{C})$ by triples $(S,D_n \times A_5, \phi) \in \mathbb{CB}$, which were described in  Theorems \ref{th01} and  \ref{thExcept}.
    \item $\mathbb{Z}_n \times \bar{A}_5$, $n \ge 3$

         This group is presented in $\Cr_2(\mathbb{C})$ by pairs $(S,\mathbb{Z}_n \times \bar{A}_5) \in \mathbb{D}$ and triples $(S,\mathbb{Z}_n \times \bar{A}_5,\phi)$, which were described in Theorems \ref{K_S9},  \ref{th01}.
     \item  $\mathbb{Z}_n \times A_5$, $n \ge 3$

    This group is presented in $\Cr_2(\mathbb{C})$ by triples $(S,\mathbb{Z}_n \times A_5,\phi) \in \mathbb{CB}$, which were described in Theorem \ref{th01}.
     \item $\mathbb{Z}_2^2 \times \bar{A}_5$

     This group is presented in $\Cr_2(\mathbb{C})$ by triples $(S,\mathbb{Z}_2^2 \times \bar{A}_5, \phi) \in \mathbb{CB}$, which were described in Theorem \ref{thExcept}.
     \item $\mathbb{Z}_2^2 \times A_5$

     This group is presented in $\Cr_2(\mathbb{C})$ by triples $(S,\mathbb{Z}_2^2 \times A_5,\phi) \in \mathbb{CB}$, which were described  in Theorems \ref{th01},  \ref{thExcept} and  \ref{th5}.
     \item $\mathbb{Z}_2 \times \bar{A}_5$

      This group is presented in $\Cr_2(\mathbb{C})$ by pairs $(S,\mathbb{Z}_2 \times \bar{A}_5) \in \mathbb{D}$ and triples $(S,\mathbb{Z}_2 \times \bar{A}_5,\phi) \in \mathbb{CB}$, which were described in  Theorems \ref{K_S9},  \ref{th01},  and  \ref{th5}.
        \item $\mathbb{Z}_2 \times A_5$

     This group is presented in $\Cr_2(\mathbb{C})$ by pairs $(S,\mathbb{Z}_2 \times A_5) \in \mathbb{D}$ and triples $(S,\mathbb{Z}_2 \times A_5,\phi) \in \mathbb{CB}$, which were described in Theorems \ref{K_S8}, \ref{th01},  \ref{th1},  \ref{th2}, \ref{th4}.
       \item $\bar{A}_5$

     This group is presented in $\Cr_2(\mathbb{C})$ by pairs $(S,\bar{A}_5) \in \mathbb{D}$ and triples $(S,\bar{A}_5,\phi) \in \mathbb{CB}$, which were described in Theorems
      \ref{K_S9},  \ref{th01}, \ref{th1},  \ref{th2},  \ref{th3},  \ref{th4}.

       \item $A_5$

    This group is presented in $\Cr_2(\mathbb{C})$ by pairs $(S,A_5) \in \mathbb{D}$ and triples $(S,A_5,\phi) \in \mathbb{CB}$, which were described in  Theorems
     \ref{K_S9},   \ref{thdp4},  \ref{th01}.
  \end{itemize}

          \end{document}